\nonstopmode
\documentclass[reqno,10pt]{amsart}
\usepackage{latexsym}
\usepackage{fancyhdr}
\usepackage{amsmath, amssymb}
\usepackage[ansinew]{inputenc}
\usepackage[all]{xy}
\usepackage{pdflscape}
\usepackage{longtable}
\usepackage{rotating}
\usepackage{verbatim}
\usepackage{hyperref}
\usepackage{subfigure}
\usepackage{mathrsfs}
\usepackage{cleveref}
\usepackage{mdwlist}
\usepackage{color}

\newcounter{mnotecount}[section]

\newcommand{\rmnote}[1]{}

\DeclareFontFamily{U}{mathb}{\hyphenchar\font45}
\DeclareFontShape{U}{mathb}{m}{n}{
      <5> <6> <7> <8> <9> <10> gen * mathb
      <10.95> mathb10 <12> <14.4> <17.28> <20.74> <24.88> mathb12
      }{}
\DeclareSymbolFont{mathb}{U}{mathb}{m}{n}
\DeclareFontSubstitution{U}{mathb}{m}{n}

\let\dot\relax
\DeclareMathAccent{\dot}{0}{mathb}{"39}
\let\ddot\relax
\DeclareMathAccent{\ddot}{0}{mathb}{"3A}
\let\dddot\relax
\DeclareMathAccent{\dddot}{0}{mathb}{"3B}
\let\ddddot\relax
\DeclareMathAccent{\ddddot}{0}{mathb}{"3C}

\theoremstyle{plain}
\newtheorem*{theorem*}{Theorem}
\newtheorem{theorem}{Theorem}[section]
\newtheorem*{lemma*}{Lemma}
\newtheorem{lemma}[theorem]{Lemma}
\newtheorem*{proposition*}{Proposition}
\newtheorem{proposition}[theorem]{Proposition}
\newtheorem*{corollary*}{Corollary}
\newtheorem{corollary}[theorem]{Corollary}
\newtheorem*{claim*}{Claim}

\newtheorem*{conjecture*}{Conjecture}
\newtheorem{question}[theorem]{Question}
\newtheorem*{question*}{Question}
\theoremstyle{definition}
\newtheorem*{definition*}{Definition}
\newtheorem{definition}[theorem]{Definition}
\newtheorem*{example*}{Example}

\newtheorem*{algorithm*}{Algorithm}
\newtheorem*{remark*}{Remark}
\newtheorem*{remarks*}{Remarks}
\newtheorem{remark}[theorem]{Remark}

\newtheorem*{convention*}{Convention}

\numberwithin{equation}{section}

\newcommand{\al}{\alpha}
\newcommand{\be}{\beta}

\newcommand{\vt}{\vartheta}
\newcommand{\ka}{\kappa}
\newcommand{\la}{\lambda}

\newcommand{\rh}{\rho}

\newcommand{\si}{\sigma}

\newcommand{\ta}{\tau}

\newcommand{\vh}{\varphi}

\newcommand{\ps}{\psi}
\newcommand{\om}{\omega}

\newcommand{\Ga}{\Gamma}

\newcommand{\Si}{\Sigma}

\newcommand{\C}{\mathbb{C}}

\newcommand{\N}{\mathbb{N}}

\newcommand{\R}{\mathbb{R}}

\newcommand{\cB}{\mathcal{B}}

\newcommand{\cJ}{\mathcal{J}}

\newcommand{\fM}{\mathfrak{M}}

\newcommand{\fS}{\mathfrak{S}}

\newcommand{\fW}{\mathfrak{W}}

\newcommand{\p}{\partial}

\renewcommand{\Im}{\mathrm{Im}}

\newcommand{\supp}{\on{supp}}

\newcommand{\on}{\operatorname}

\newcommand{\sr}[1]%
{\ifmmode{}^\dagger\else${}^\dagger$\fi\ifvmode
\vbox to 0pt{\vss
 \hbox to 0pt{\hskip\hsize\hskip1em
 \vbox{\hsize3cm\raggedright\pretolerance10000
 \noindent #1\hfill}\hss}\vss}\else
 \vadjust{\vbox to0pt{\vss%
 \hbox to 0pt{\hskip\hsize\hskip1em%
 \vbox{\hsize3cm\raggedright\pretolerance10000%
 \noindent #1\hfill}\hss}\vss}}\fi%
}

\newcommand{\RR}{\mathbb R}
\newcommand{\NN}{\mathbb N}

\newcommand{\A}{\;\forall}
\newcommand{\E}{\;\exists}

\newcommand{\ol}{\overline}
\newcommand{\ul}{\underline}

\title[On the extension of Whitney ultrajets]
{On the extension of Whitney ultrajets}

\author[A.~Rainer]{Armin Rainer}

\address{A.~Rainer: 
Fakult\"at f\"ur Mathematik, Universit\"at Wien, 
Oskar-Morgenstern-Platz~1, A-1090 Wien, Austria}
\email{armin.rainer@univie.ac.at}

\author[G.~Schindl]{Gerhard Schindl}

\address{G.~Schindl: Departamento de \'Algebra, An\'alisis Matem\'atico, Geometr\'{\i}a y Topolog\'{\i}a,
Facultad de Ciencias,
Universidad de Valladolid,
Paseo de Bel\'en 7,
47011 Valladolid, Spain}
\email{gerhard.schindl@univie.ac.at}

\begin{document}

\begin{abstract}
	We prove necessary and sufficient conditions for the validity of Whitney's extension 
	theorem in the ultradifferentiable Roumieu setting with controlled loss of regularity.       
\end{abstract}

\thanks{Supported by the FWF-Projects P~26735-N25 and J 3948-N35}
\date{September 4, 2017}

\maketitle

\section{Introduction}

Whitney's extension theorem \cite{Whitney34a} 
provides conditions for the extension of jets defined in closed subsets of $\R^n$ to 
infinitely differentiable functions on $\R^n$. Its ultradifferentiable analogues ask for a 
precise determination how the growth rate of the jets is preserved by their extension. 
The growth rate of the jets, respectively of the derivatives of a smooth function, is 
measured by weight functions $\om$. 
We denote by $\cB^{\{\om\}}(\R^n)$ the associated space of ultradifferentiable functions $f$ 
on $\R^n$; by definition, the growth rate of the sequence $(\|f^{(\al)}\|_{L^\infty(\R^n)})_{\al \in \N^n}$ is 
regulated in terms of $\om$. 
We use the letter $\cB$ to emphasize that the bounds are global in $\R^n$. 
These classes of ultradifferentiable functions were introduced by Beurling \cite{Beurling61} and Bj\"orck \cite{Bjoerck66} 
and equivalently described by Braun, Meise, and Taylor \cite{BMT90}.
Similarly, $\cB^{\{\om\}}(E)$ is the space of jets on the compact subset $E \subseteq \R^n$ 
with a growth rate regulated by $\om$, so-called \emph{ultrajets}.
Precise definitions will be given in \Cref{sec:spaces}. 

The weight functions $\om$ which allow for an extension theorem preserving the class $\cB^{\{\om\}}$ have been 
fully characterized. We denote by $j_E^\infty$ the mapping which sends a smooth function 
to the infinite jet consisting of its partial derivatives of all orders restricted to $E$. 

\begin{theorem} \label{thm:preservingclass}
  Let $\om$ be a weight function. The following conditions are equivalent:
  \begin{enumerate}
    \item For every compact $E \subseteq \R^n$ the jet mapping $j^\infty_E : \cB^{\{\om\}}(\R^n) \to \cB^{\{\om\}}(E)$ 
    is surjective.
    \item There is a compact $E \subseteq \R^n$ such that $j^\infty_E : \cB^{\{\om\}}(\R^n) \to \cB^{\{\om\}}(E)$ 
    is surjective. 
    \item $\om$ is \emph{strong}, i.e., $\int_{1}^\infty \frac{\om(tu)}{u^2}\,du \le C\om(t) + C$ for all $t>0$ and some 
    $C>0$.
  \end{enumerate}
\end{theorem}

Note that a strong weight function is necessarily non-quasianalytic.
\Cref{thm:preservingclass} is due to Bonet, Braun, Meise, and Taylor \cite{BBMT91} and Abanin \cite{Abanin01} 
(the latter showed the equivalence with (2)). 
Partial results have been contributed in earlier papers, e.g.\ Meise and Taylor \cite{MeiseTaylor88}, 
Bonet, Meise, and Taylor \cite{BonetMeiseTaylor89a}. 
We want to mention that the statement remains true if the Roumieu type classes $\cB^{\{\om\}}$ are replaced by 
the Beurling type classes $\cB^{(\om)}$, but we shall only be concerned with the Roumieu case in this paper.

The purpose of this paper is to study the extension problem for weight functions $\om$ which are not strong.
In that case the extension involves a loss of regularity: the class is not preserved.  
So we are led to the following problem.

\begin{question} \label{Q}
  Let $\om$ be a non-quasianalytic weight function. 
  Let $\si$ be another weight function. 
  Under which conditions is the jet mapping $j^\infty_E$ defined on $\cB^{\{\om\}}(\R^n)$
  surjective onto $\cB^{\{\si\}}(E)$ for all compact $E \subseteq \R^n$?
\end{question}

A complete answer has been given for the one-point set $E=\{0\}$, by Bonet, Meise, and Taylor \cite{BonetMeiseTaylor92}, 
and for compact convex sets 
$E$, by Langenbruch \cite{Langenbruch94}. 
In these cases the mapping $j^\infty_{E} : \cB^{\{\om\}}(\R^n) \to \cB^{\{\si\}}(E)$ is surjective if and only if 
\begin{equation} \label{intro:mixed}
  \E C>0 \A t > 0 : \int_{1}^\infty \frac{\om(tu)}{u^2}\,du \le C\si(t) + C.   
\end{equation} 
So this condition is necessary for our problem.

We answer \Cref{Q} (for all compact $E \subseteq \R^n$) under three additional conditions. 
The first condition is that $\om$ is concave. This has technical reasons, but it is not incongruous, since every strong 
weight function is equivalent to a concave one; cf.\ \cite[Proposition 1.3]{MeiseTaylor88}. 
Secondly, we require that $\si(t) = o(t)$ as $t \to \infty$; again any strong weight function has this property. 

To explain the third condition let us recall that any weight function $\si$ is associated with a family of weight sequences 
$\fS = \{S^x\}_{x>0}$ such that for the corresponding ultradifferentiable spaces we have 
\begin{equation} \label{intro:rep}
    \cB^{\{\si\}}(\R^n) = \on{ind}_{x>0} \cB^{\{S^x\}}(\R^n) 
    \quad (\text{and } \cB^{(\si)}(\R^n) = \on{proj}_{x>0} \cB^{(S^x)}(\R^n)).
\end{equation}
The condition we require is that 
  \begin{equation} \label{intro:good}
    \A x >0 \E y>0 \E C \ge 1 \A 1\le j \le k : \frac{S^{x}_j}{jS^x_{j-1}} \le C\, \frac{S^{y}_k}{k S^y_{k-1}}. 
  \end{equation}
The following is our main result.

\begin{theorem} \label{main}
Let $\om$ be a non-quasianalytic concave weight function. 
Let $\si$ be a weight function satisfying $\si(t) = o(t)$ as $t \to \infty$ and \eqref{intro:good}. 
Then the following conditions are equivalent:
\begin{enumerate}
    \item For every compact $E \subseteq \R^n$ the jet mapping $j^\infty_E : \cB^{\{\om\}}(\R^n) \to \cB^{\{\si\}}(E)$ 
    is surjective. 
    \item There is $C>0$ such that $\int_{1}^\infty \frac{\om(tu)}{u^2}\,du \le C\si(t) + C$ for all $t>0$.
  \end{enumerate}
\end{theorem}

The implication (1) $\Rightarrow$ (2) follows from the aforementioned result of 
\cite{BonetMeiseTaylor92} and does not require the three 
additional conditions on $\om$ and $\si$. We discuss the condition \eqref{intro:good} and its relation to other 
properties of the weight function in \Cref{sec:good} and \Cref{sec:strong}. 
Let us emphasize that, while (1) and (2) in \Cref{main} are invariant under equivalence of weight functions (two weight 
functions are equivalent if and only if they generate the same class),  
concavity and \eqref{intro:good} are not invariant. Thus, for the validity of \Cref{main} is is enough that the 
assumptions on $\om$ and $\si$ are satisfied up to equivalence of weight functions.

The problem put forward in \Cref{Q} has been solved for Denjoy--Carleman classes by Chaumat and Chollet 
\cite{ChaumatChollet94}, where the 
growth rate of the derivatives is controlled by weight sequences $M$. 
Indeed, under suitable conditions on the weight sequences, 
\cite{ChaumatChollet94} proved that 
the jet mapping $j^\infty_E : \cB^{\{N\}}(\R^n) \to \cB^{\{M\}}(E)$ is surjective, for every compact $E \subseteq \R^n$, 
if and only if 
\[
  \E C>0 \A k \in \N : \sum_{\ell \ge k} \frac{N_{\ell -1}}{N_\ell} \le C \, \frac{k M_{k-1}}{M_k}.
\] 
The case that the extension preserves the class (i.e., $M=N$) is due to Bruna \cite{Bruna80} 
(see also \cite{Lambert79}).
We will see that our \Cref{main} is a generalization of this result (under an additional assumption on $N$). In general, 
a class $\cB^{\{\om\}}$ cannot be represented as a class $\cB^{\{M\}}$ for a weight sequence $M$, and vice versa, 
cf.\ Bonet, Meise, and Melikhov\cite{BMM07} and Rainer and Schindl \cite{RainerSchindl12}.

The approach of \cite{ChaumatChollet94} was the starting point of our recent paper \cite{RainerSchindl16a} in which we 
obtained a generalization of their extension result for \emph{admissible} unions of Denjoy--Carleman classes. 
By virtue of \eqref{intro:rep}, we deduced a version of \Cref{main} which however required an restrictive undesired 
condition on the involved weight functions.  

In the present paper we surmount this problem by using the special cut-off functions which were constructed in 
\cite{BBMT91}. They are tailor-made for weight functions $\om$; we actually need a modified version 
for two weight functions $\om$ 
and $\si$ related by \eqref{intro:mixed}. Then we combine the resulting partition of unity $\{\vh_i\}$ subordinate to 
a collection of Whitney cubes $Q_i$ with center $x_i$ with the technique of 
\cite{ChaumatChollet94} which is based on a extension method of Dynkin \cite{Dynkin80}. 
The extension of a ultrajet $F$ of class $\cB^{\{\si\}}$ is defined as a linear combination 
$$\sum_i \vh_i\, T_{\hat x_i}^{p(x_i)} F$$
of Taylor polynomials, where the degree $p(x_i)$
depends on the distance of $x_i$ to $E$ and $\hat x_i \in E$ realizes this distance. 
More precisely, the dependence of $p$ is through counting functions corresponding to the sequences in $\fS$, 
the family associated with $\si$. It is this part of the proof which necessitates the assumption \eqref{intro:good}.

The paper is structured as follows.
We introduce weight functions, weight sequences, and the corresponding spaces of ultradifferentiable functions 
and jets in \Cref{sec:spaces}. 
A deeper analysis of the weights, their associated functions, and properties needed in the proof of the extension theorem
follows in \Cref{sec:moreonweights}. 
We recall
the construction of special cut-off functions due to \cite{BBMT91} in \Cref{sec:partition}; 
since we need a slight generalization for two weight functions $\om$ and $\si$ satisfying \eqref{intro:mixed}, 
we indicate the required modifications in the proof.
The main theorem \ref{main} and its corollaries are proved in \Cref{sec:extension}.

\section{Spaces of ultradifferentiable functions and jets} \label{sec:spaces}

\subsection{Weight functions} \label{weightfunction}

By a \emph{weight function} we mean a continuous increasing function $\om : [0,\infty) \to [0,\infty)$ with $\om(0) =0$ 
and 
$\lim_{t \to \infty} \om(t) = \infty$ that satisfies
\begin{align}
   & \om(2t) = O(\om(t)) \quad\text{ as } t \to \infty, \label{om1}\\
   & \om(t) = O(t) \quad\text{ as } t \to \infty, \label{om2}\\
   & \log t = o(\om(t)) \quad\text{ as } t \to \infty, \label{om3}\\
   & \vh(t) := \om(e^t) \text{ is convex}.  \label{om4}
\end{align} 
A weight function is called \emph{non-quasianalytic} if
\begin{equation}
   \int_0^\infty \frac{\om(t)}{1+t^2} \, dt <\infty.
 \end{equation}
Two weight functions $\om$ and $\si$ are said to be \emph{equivalent} if $\om(t) = O(\si(t))$ and $\si(t) = O(\om(t))$ 
as $t \to \infty$. 
For each weight function $\om$ there is an equivalent weight function 
$\tilde \om$ such that $\om(t) = \tilde \om(t)$ for large $t>0$ 
and $\tilde \om |_{[0,1]} =0$. It is thus no restriction to assume that $\om |_{[0,1]} =0$ when necessary.

The \emph{Young conjugate} $\vh^*$ of $\vh$ is defined by  
\[
  \vh^*(t) := \sup_{s\ge 0} \big(st-\vh(s)\big), \quad t \ge 0.
\]
Assuming $\om |_{[0,1]} =0$, we have that    
$\vh^*$ is a convex increasing function satisfying $\vh^*(0)=0$, $t/\vh^*(t) \to 0$ as $t \to \infty$, and $\vh^{**}=\vh$; 
cf.\ \cite{BMT90} and \cite[Remark 1.2]{BBMT91}.

\subsection{The space \texorpdfstring{$\cB^{\{\om\}}(\R^n)$}{} of ultradifferentiable functions} 

Let $\om$ be a weight function and $\rh >0$.  
We consider the Banach space $\cB^{\om}_\rh(\R^n) := \{f \in C^\infty (\R^n) : \|f\|^\om_{\rh}< \infty\}$, 
where  
\[
  \|f\|^\om_{\rh} := \sup_{x \in \R^n,\,\al \in \N^n} |\p^\al f(x)| \exp(-\tfrac{1}{\rh} \vh^*(\rh |\al|)),
\]
and the inductive limit 
\begin{equation*}
  \cB^{\{\om\}}(\R^n) := \on{ind}_{\rh \in \N} \cB^{\om}_\rh(\R^n).
\end{equation*}
For weight functions $\om$ and $\si$ we have $\cB^{\{\om\}} \subseteq \cB^{\{\si\}}$ if and only if
$\si(t) = O(\om(t))$ as $t \to \infty$,
cf.\ \cite[Corollary 5.17]{RainerSchindl12}; in particular, $\om$ and $\si$ are equivalent if and only if 
$\cB^{\{\om\}} = \cB^{\{\si\}}$.  
The space $\cB^{\{\om\}}(\R^n)$ contains non-trivial functions with compact support if and only if $\om$ is 
non-quasianalytic (cf.\ \cite{BMT90} or \cite{RainerSchindl12}).

\subsection{Weight sequences} \label{weights}

Let $\mu = (\mu_k)$ be a positive increasing sequence,  
$1 = \mu_0 \le \mu_1 \le \mu_2 \le \cdots$.
We associate the sequences $M=(M_k)$ and $m = (m_k)$ defined by 
\begin{equation} \label{def:M}
  \mu_0 \mu_1 \mu_2 \cdots \mu_k = M_k = k!\, m_k, 
\end{equation}
for all $k \in \N$. 
We call $M$ a \emph{weight sequence} if $M_k^{1/k} \to \infty$.    
A weight sequence $M$ is called \emph{non-quasianalytic} if 
\begin{equation}
 \sum_{k} \frac{1}{\mu_k} < \infty. 
\end{equation}
We say that $M$ has \emph{moderate growth} if there exists $C>0$ such that $M_{j+k} \le C^{j+k} M_j M_k$ for all
$j,k \in \N$, or equivalently,
\begin{equation}
  \mu_k \lesssim M_k^{1/k};
\end{equation}
we refer to \cite[Lemma 2.2]{RainerSchindl16a} for a proof and more equivalent conditions.
(For real valued functions $f$ and $g$ we write $f \lesssim g$ if $f \le C g$ for some positive constant $C$.)

Two weight sequences $M$ and $N$ are said to be \emph{equivalent} if there is a constant $C>0$ such that 
$1/C \le M_k^{1/k}/N_k^{1/k} \le C$ for all $k$. 

\begin{remark}
  (1) Some authors (e.g.\ \cite{ChaumatChollet94}, \cite{RainerSchindl12}) 
  prefer to work with \emph{``sequences without factorials''}, that is 
  $m_k$ instead of $M_k$. 

  (2) Note that $\mu$ uniquely determines $M$ and $m$, and vice versa. In analogy we shall use 
$\nu \leftrightarrow N \leftrightarrow n$, $\si \leftrightarrow S \leftrightarrow s$, etc.
That $\mu$ is increasing means precisely that $M$ is logarithmically convex (\emph{log-convex} for short). 
Log-convexity of $m$ is a stronger condition: if $m$ is log-convex we shall say that $M$ is 
\emph{strongly log-convex}. 
\end{remark}

\begin{lemma}[Properties of weight sequences] \label{lem:basicM}
Let $1 = \mu_0 \le \mu_1 \le \mu_2 \le \cdots$. Then:
  \begin{enumerate}
      \item $M_k^{1/k}$ is increasing, equivalently,
      \begin{equation}\label{mucompare}
        \A k\in \N_{>0} : M_k^{1/k} \le \mu_k.
      \end{equation}
      \item $M_jM_k\le M_{j+k}$ for all $k,j$. 
      \item If $M_k^{1/k} \to \infty$, then $\mu_k \to \infty$.         
  \end{enumerate}  
\end{lemma}

\begin{proof}
  This is straightforward to check.
\end{proof}

\subsection{The space \texorpdfstring{$\cB^{\{M\}}(\R^n)$}{} of ultradifferentiable functions} 

Let $M=(M_k)$ be a weight sequence and $\rh >0$.
We consider the Banach space $\cB^{M}_\rh(\R^n) := \{f \in C^\infty (\R^n) : \|f\|^M_{\rh}< \infty\}$, where 
\[
  \|f\|^M_{\rh} := \sup_{x \in \R^n,\, \al \in \N^n} \frac{|\p^\al f(x)|}{\rh^{|\al|} M_{|\al|}},
\]
and the inductive limit 
\begin{equation*} 
  \cB^{\{M\}}(\R^n) := \on{ind}_{\rh \in \N} \cB^{M}_\rh(\R^n).  
\end{equation*}
Traditionally, $\cB^{\{M\}}(\R^n)$ is called a \emph{Denjoy--Carleman class}.
For weight sequences $M$ and $N$ we have $\cB^{\{M\}} \subseteq \cB^{\{N\}}$ if and only if 
$M_k^{1/k} \lesssim N_k^{1/k}$; one implication is obvious, the other follows from the existence of 
\emph{characteristic} $\cB^{\{M\}}$-functions, cf. \cite[Lemma 2.9 and Proposition 2.12]{RainerSchindl12}.
In particular, $M$ and $N$ are equivalent if and only if the corresponding classes coincide.
By the Denjoy--Carleman theorem 
(e.g.\ \cite[Theorem 1.3.8]{Hoermander83I}),
$\cB^{\{M\}}(\R^n)$ contains non-trivial elements 
with compact support if and only if $M$ is non-quasianalytic.

\subsection{The connection between \texorpdfstring{$\cB^{\{\om\}}(\R^n)$}{} and \texorpdfstring{$\cB^{\{M\}}(\R^n)$}{}} 

With any weight function $\om$ we can associate a family of weight sequences $\{W^x\}_{x>0}$ such that 
$\cB^{\{\om\}}(\R^n)$ can be described us the union of the spaces $\cB^{\{W^x\}}(\R^n)$; see \Cref{representation} below.
 
\begin{definition}[The weight matrix associated with a weight function]
With a weight function $\om$ we associate a \emph{weight matrix} $\fW = \{W^x\}_{x>0}$ by setting 
\begin{equation*}
  W^x_k := \exp(\tfrac{1}{x}\vh^*(x k)), \quad k \in \N;
\end{equation*}
cf.\ \cite[5.5]{RainerSchindl12}. 
Moreover, we define
\begin{equation*}
  \vt^x_k := \frac{W^x_k}{W^x_{k-1}}. 
\end{equation*}  
\end{definition}

\begin{lemma}[Properties of the associated weight matrix] \label{lemma4} 
  We have:
  \begin{enumerate}
     \item Each $W^x$ is a weight sequence (in the sense of \Cref{weights}).
     \item $\vt^x \le \vt^y$ if $x \le y$, which entails $W^x \le W^y$. \label{lemma4(2)}
     \item For all $x>0$ and all $j,k \in \N$,  $W^x_{j+k} \le W^{2x}_{j} W^{2x}_{k}$ and $w^x_{j+k} \le w^{2x}_{j} w^{2x}_{k}$.
     \label{eq:mW}
     \item For all $x>0$ and all $k \in \N_{\ge 2}$, $\vt^x_{2k} \le   \vt^{4x}_{k}$. 
     \item $\A \rh>0 \E H\ge 1 \A x >0 \E C \ge 1 \A k \in \N : \rh^k W^x_k \le C W^{Hx}_k$. \label{5.10}
     \item If $\om(t) = o(t)$ as $t \to \infty$ then $(w^x_k)^{1/k} \to \infty$ and $\vt^{x}_k/k \to \infty$ for all $x>0$.
   \end{enumerate} 
\end{lemma}

\begin{proof} 
  (1)--(3) These are direct consequences of the properties of $\vh^*$; cf.\ \cite[5.5]{RainerSchindl12}.

  (4) \cite[Lemma 2.6]{RainerSchindl16a}.

  (5) \cite[Lemma 5.9]{RainerSchindl12}.

  (6) By \cite[Corollary 5.15]{RainerSchindl12}, we have $(w^x_k)^{1/k} \to \infty$. 
  That also $\vt^x_k/k \to \infty$ follows from \eqref{mucompare}. 
\end{proof}

\begin{theorem}[{\cite[Corollaries 5.8 and 5.15]{RainerSchindl12}}] \label{representation}
  Let $\om$ be a weight function and let $\fW = \{W^x\}_{x>0}$ be the associated weight matrix. 
  Then, as locally convex spaces,
  \begin{align} \label{eq:repR}
    \cB^{\{\om\}}(\R^n) &= \on{ind}_{x > 0} \cB^{\{W^x\}}(\R^n) = \on{ind}_{x > 0} \on{ind}_{\rh>0} \cB^{W^x}_\rh(\R^n).  
  \end{align}
  We have $\cB^{\{\om\}}(\R^n) = \cB^{\{W^x\}}(\R^n)$ for all $x>0$ if and only if 
  \begin{align}
     \E H\ge 1 \A t\ge 0 : 
      2\om(t) \le \om(Ht) + H.   \label{om6}
  \end{align} 
  Moreover, \eqref{om6} holds if and only if some (equivalently each) $W^x$ has moderate growth. 
  It is no restriction to let the inductive limits in \eqref{eq:repR} range only over $x, \rh \in \N$. 
\end{theorem}

\begin{remark}
  Let us emphasize that the fact that $\cB^{\{\om\}} = \cB^{\{M\}}$ for some weight sequence $M$ 
  if and only if $\om$ satisfies \eqref{om6} is due to \cite{BMM07}.
\end{remark}

\subsection{Whitney ultrajets}

Let $E$ be a compact subset of $\R^n$. 
We denote by $\cJ^\infty(E)$ the vector space of all jets $F= (F^\al)_{\al\in \N^n} \in C^0(E,\R)^{\N^n}$ on $E$. 
For $a \in E$ and $p \in \N$ we associate the Taylor polynomial
\begin{gather*}
  T^p_a : \cJ^\infty(E) \to C^\infty (\R^n,\R), ~ F \mapsto T^p_a F(x) := \sum_{|\al|\le p} \frac{(x-a)^\al}{\al!} F^\al(a),    
\end{gather*}
and the remainder $R^p_a F = ((R^p_a F)^\al)_{|\al| \le p}$ with
\begin{gather*}
  (R^p_a F)^\al (x) := F^\al(x) - \sum_{|\be| \le p-|\al|} \frac{(x-a)^\be}{\be!} F^{\al+\be}(a), \quad a,x \in E.      
\end{gather*}
Let us denote by 
$j^\infty_E$ the mapping which assigns to a $C^\infty$-function $f$ on $\R^n$
the jet $j^\infty_E(f) := (\p^\al f|_E)_\al$. 
By Taylor's formula, $F  =j^\infty_E(f)$ satisfies
\begin{equation*}\label{Whitneyjets}
  (R^p_a F)^\al (x) = o(|x-a|^{p-|\al|}) \quad \text{ for $a,x \in E$, $p\in \N$, $|\al| \le p$ as $|x-a|\to 0$.}
\end{equation*}
Conversely, if a jet $F \in \cJ^\infty(E)$ has this property, then it admits a 
$C^\infty$-extension to $\R^n$,
by Whitney's extension theorem \cite{Whitney34a} (for modern accounts see e.g.\ 
\cite[Ch.~1]{Malgrange67}, \cite[IV.3]{Tougeron72}, or 
\cite[Theorem 2.3.6]{Hoermander83I}).

\begin{definition}[Whitney ultrajets] \label{def:ultrajets}
Let $E \subseteq \R^n$ be compact.
Let $M=(M_k)$ be a weight sequence.
For fixed $\rh>0$ we denote by $\cB^{M}_\rh(E)$ the set of all jets $F$ such that  
there exists $C>0$ with 
\begin{gather*}
  |F^\al(a)| \le C \rh^{|\al|} \,  M_{|\al|}, \quad \al \in \N^n,~ a \in E,
  \\
  |(R^p_a F)^\al(b)| \le C \rh^{p+1} \, M_{p+1}\,  \frac{|b-a|^{p+1-|\al|}}{(p+1-|\al|)!}, \quad p \in \N,\, |\al| \le p,~ a,b \in E.  
\end{gather*}
The smallest constant $C$ defines a complete norm on $\cB^{M}_\rh(E)$. We define 
$$\cB^{\{M\}}(E) := \on{ind}_{\rh \in \N} \cB^{M}_\rh(E).$$
An element of $\cB^{\{M\}}(E)$ is called a \emph{Whitney ultrajet of class $\cB^{\{M\}}$ on $E$}.

Let $\om$ be a weight function and $\fW = \{W^x\}_{x>0}$ the associated weight matrix.
A jet $F$ is said to be a \emph{Whitney ultrajet of class $\cB^{\{\om\}}$ on $E$}  
if $F \in \cB^{\{W^x\}}(E)$ 
for some $x>0$; we set 
$$\cB^{\{\om\}}(E)= \cB^{\{\fW\}}(E)= \on{ind}_{x >0} \cB^{\{W^x\}}(E) = \on{ind}_{x>0} \on{ind}_{\rh>0} \cB^{W^x}_\rh(E).$$ 
\end{definition}

\begin{remark}
  This definition of Whitney ultrajet of class $\cB^{\{\om\}}$ on $E$ coincides with the one given in 
  \cite{BBMT91}. This follows from \Cref{lemma4}\eqref{5.10}.
\end{remark}

\subsection{Notation for sequences}
\hfill
\smallskip

\begin{minipage}{0.6\textwidth}
	The table summarizes our notation for sequences appearing in the paper. 
	The three columns are mutually determined by the rule
	\[		
		\mu_0 \mu_1 \mu_2 \cdots \mu_k = M_k = k!\, m_k
	\]
	for $k \in \N$.
	(There will be no confusion by the fact that $\si$ usually denotes a weight function.)
\end{minipage}
\hspace{12mm}
\begin{minipage}{0.3\textwidth}
\begin{tabular}[h]{|c|c|c|}
	$M$ & $m$ & $\mu$ \\
	$N$ & $n$ & $\nu$ \\
	$L$ & $\ell$ & $\la$ \\
	$W^x$ & $w^x$ & $\vt^x$ \\
	$S$ & $s$ & $\si$ \\
	$\dot S$ & $\dot s$ & $\dot \si$
\end{tabular}	
\end{minipage}

\section{More on weight functions and weight sequences} \label{sec:moreonweights}

\subsection{Functions associated with weight sequences} \label{hGaSi}

There are a few functions which one naturally associates with a weight sequence; cf.\ 
\cite{Mandelbrojt52}, \cite{Komatsu73}, \cite{ChaumatChollet94}. They will play an essential role in the proof 
of the extension theorem \ref{main}.

\begin{definition}[Associated functions]
Let $m =(m_k)$ be a positive sequence satisfying $m_0 = 1$ and $m_k^{1/k} \to \infty$ (not necessarily log-convex). 
We associate the following functions  
\begin{align} \label{h}
  h_m(t) &:= \inf_{k \in \N} m_k t^k, \quad t > 0, \quad h_m(0):=0, \\
  \label{counting2}
  \ol \Ga_m(t) &:= \min\{k : h_m(t) =  m_k t^k\}, \quad t > 0,  \\
\intertext{and, provided that $m_{k+1}/m_{k} \to \infty$,}
\ul \Ga_m (t) &:=  \min\Big\{k : \frac{m_{k+1}}{m_k}  \ge \frac{1}{t} \Big\}, \quad t > 0.
\end{align}
\end{definition}

\begin{lemma} \label{basic}
Let $m =(m_k)$ be a positive sequence satisfying $m_0 = 1$, $m_k^{1/k} \to \infty$, and $m_{k+1}/m_k \to \infty$. Then: 
\begin{enumerate}
  \item $h_m$ is increasing, continuous, and positive for $t>0$. For large $t$ we have $h_m(t) = 1$.
  \item $\ul \Ga_m$ is decreasing 
  and $\ul \Ga_m(t) \to \infty$ as $t\to 0$.
  \item $k \mapsto m_k t^k$ is decreasing for $k \le \ul \Ga_m(t)$. \label{eq:ulGa3}
  \item $m_{k+1}/m_k  \le n_{k+1}/n_k$ for all $k$ implies $\ul \Ga_n \le \ul \Ga_m$.
  \item $\ul \Ga_m  
  \le \ol \Ga_m$. If $m$ is log-convex then $\ul \Ga_m = \ol \Ga_m$.
\end{enumerate}  
\end{lemma}

\begin{proof}
  These facts are well-known and immediate from the definitions; we refer to \cite{Mandelbrojt52}, \cite{Komatsu73}, 
  and \cite{ChaumatChollet94}.
\end{proof}

Let $M$ be a weight sequence satisfying $m_k^{1/k} \to \infty$. 
Then $m_{k}/m_{k-1} = \mu_k/k \to \infty$, in fact, 
we have $(k!\, m_k)^{1/k} = M_k^{1/k} \le \mu_k$, by \eqref{mucompare}.

So for such $M$ the functions $h_m$, $\ul \Ga_m$, $\ol \Ga_m$   
are well-defined and enjoy the properties listed 
in \Cref{basic}.  
The sequence $m$ \emph{will not be log-convex} in general, whence $\ul \Ga_m$ and $\ol \Ga_m$ fall apart. 
We need them both. 
It will crucial to be able to compare them, which is the content of the following lemma.
Of course, we pay the price that we must switch from $m$ to another sequence $n$.

\begin{lemma} \label{lem:m1}
  Let $M$, $N$ be weight sequences satisfying $m_k^{1/k} \to \infty$ and $n_k^{1/k} \to \infty$.
  Assume that there exists $C\ge 1$ such that
  $\mu_j/j \le C \nu_k/k$ for all $j \le k$.
  Then, for all $t >0$,
  \begin{equation} \label{eq:compare} 
    \ol \Ga_n (Ct) \le \ul \Ga_m(t). 
  \end{equation}
\end{lemma}

\begin{proof}
	Let $t >0$.				
  If $k > \ul \Ga_m(t)$, then
  \begin{align*}
    \frac{n_k}{n_{\ul \Ga_m(t)}} = \frac{\nu_k}{k} \cdots \frac{\nu_{\ul \Ga_m(t)+1}}{\ul \Ga_m(t)+1} 
    \ge \Big(\frac{\mu_{\ul \Ga_m(t)+1}}{C (\ul \Ga_m(t)+1) }\Big)^{k - \ul \Ga_m(t)} \ge (Ct)^{-k + \ul \Ga_m(t)}
  \end{align*}
  and thus $n_k (Ct)^k \ge n_{\ul \Ga_m(t)} (Ct)^{\ul \Ga_m(t)}$. 
  It follows that $\ol \Ga_n (Ct) \le \ul \Ga_m(t)$. 
\end{proof}

We also need the following property.

\begin{lemma} \label{lem:m2} 
  Let $M$, $N$, $L$ be weight sequences satisfying $m_k^{1/k} \to \infty$, $n_k^{1/k} \to \infty$, 
  and $\ell_k^{1/k} \to \infty$. Assume that  
  \begin{equation} \label{eq:m1}
    \mu_{2k} \lesssim \nu_k
  \end{equation}
  and 
  \begin{equation} \label{eq:m2}
    \E C \ge 1 \A 1 \le j \le k  : \frac{\nu_j}{j}  \le C \frac{\la_k}{k}.
  \end{equation}
  Then 
  \begin{equation} \label{eq:m3}
    \E D \ge 1 \A t >0 : 2 \ul \Ga_\ell(Dt) \le \ul \Ga_m(t).
  \end{equation}
\end{lemma}

\begin{proof}
  We first claim that \eqref{eq:m1} and \eqref{eq:m2} imply 
  \begin{equation} \label{eq:m4}
    \E C \ge 1 \A 1 \le h \le 2k  : \frac{\mu_h}{h} \le C \frac{\la_k}{k}.
  \end{equation}
  Note that \eqref{eq:m1} is equivalent to $\frac{\mu_{2k}}{2k} \lesssim \frac{\nu_k}{k}$.
  Thus, if $h = 2j$ for  $1 \le j \le k$, then $\frac{\mu_{h}}{h} = \frac{\mu_{2j}}{2j} \le C \frac{\la_k}{k}$. 
  If $h$ is odd, then $\frac{\mu_h}{h} \le 2 \frac{\mu_{h+1}}{h+1} \le 2C \frac{\la_k}{k}$, since $\mu$ is increasing. 

  Now it is easy to see that \eqref{eq:m4} implies \eqref{eq:m3}.
\end{proof}

\subsection{Good weight functions}

Let us single out the weight functions whose associated weight matrix 
satisfies the conditions required in \Cref{lem:m1} and \Cref{lem:m2}.

\begin{definition}[Good weight functions] \label{def:good}
  A weight function $\om$ with associated weight matrix $\fW = \{W^x\}_{x >0}$ is called \emph{good} if
  \begin{equation} \label{eq:good}
    \A x >0 \E y>0 \E C \ge 1 \A 1\le j \le k : \frac{\vt^{x}_j}{j} \le C\, \frac{\vt^{y}_k}{k}. 
   \end{equation} 
\end{definition}

\begin{remark} \label{rem:good}
  (1) By \Cref{lemma4}\eqref{lemma4(2)}, it is no restriction to assume $y \ge 2x$ in \eqref{eq:good} 
  to the benefit that $w^x_{j+k} \le w^{y}_j w^{y}_k$ for all $j,k$, by \Cref{lemma4}\eqref{eq:mW}. 

  (2) We remark that \eqref{eq:good} is not invariant under equivalence of weight functions.  

  (3) If $W^x$ is strongly log-convex, then \eqref{eq:good} is satisfied with $y = x$ and $C =1$. 
\end{remark}

\begin{proposition} \label{cor:essence}
  Let $\om$ be a good weight function satisfying $\om(t) = o(t)$ as $t \to \infty$.
  Let $\fW = \{W^x\}_{x>0}$ be the associated weight matrix.
  Then 
  \begin{align} \label{eq:essence1}
    \A x >0 \E &y_3 \ge y_2 \ge y_1 \ge x \E D\ge 1 \A t>0 : 
    \notag \\
    &
     \ol \Ga_{w^{y_3}}(D^3t)\le \ul \Ga_{w^{y_2}}(D^2t)  
     \le \ol \Ga_{w^{y_2}}(D^2t) 
    \le \ul \Ga_{w^{y_1}}(Dt) \le \frac{\ul \Ga_{{w^x}}(t)}{2}.
  \end{align}
  We may assume that $y_1 \ge 2x$ and $y_2 \ge 2y_1$ and hence $w^x_{j+k} \le w^{y_1}_j w^{y_1}_k$ and 
  $w^{y_1}_{j+k} \le w^{y_2}_j w^{y_2}_k$ for all $j,k \in \N$. 
\end{proposition}

\begin{proof}
  The rightmost inequality in \eqref{eq:essence1}
  follows from \Cref{lem:m2}, since 
  $\vt^x_{2k} \le \vt^{4x}_{k}$ (for $k \ge 2$) by \Cref{lemma4}(4). 
  The other inequalities are easy consequences of \Cref{lem:m1} and \Cref{basic}(5).
  The supplement follows from \Cref{rem:good}(1). 
\end{proof}

\subsection{The conjugate of a weight function}

The following conjugate will be important for the special partition of 
unity to be constructed in \Cref{sec:partition}.

\begin{definition}[The conjugate of a weight function]
Let $\omega : [0,\infty) \to [0,\infty)$ satisfy $\om(t) = o(t)$ as $t \to \infty$. We define
\begin{equation}\label{omegaconjugate}
\omega^{\star}(t):=\sup_{s\ge 0} \big(\omega(s)-st\big), \quad t>0.
\end{equation}  
\end{definition}

Then $\omega^\star$ is decreasing, continuous, and convex
with $\omega^{\star}(t) \to \infty$ as $t \to 0$, see \cite[Remark 1.5]{PetzscheVogt84}.
Since $\om(t) = o(t)$ as $t \to \infty$, $\om^\star(t)$ is finite for all $t$. 
If $\omega$ is concave and increasing, then, by \cite[Proposition 1.6]{PetzscheVogt84},
\begin{equation}\label{omegaconjugate1}
\omega(t)=\inf_{s>0} \big(\omega^{\star}(s)+st\big), \quad t>0.
\end{equation}

\begin{lemma} \label{lem:omsistar}
  Let $\om,\si : [0,\infty) \to [0,\infty)$ satisfy $\om(t) = o(t)$ and $\si(t) = o(t)$ as $t \to \infty$.
  Suppose that $\si(t) = O(\om(t))$ as $t \to \infty$. Then 
  \begin{equation}
    \E C \ge 1 \A t>0 : \si^\star (t) \le C \om^\star(t/C) + C.
  \end{equation}
\end{lemma}

\begin{proof}
  This is an easy computation.
\end{proof}

\subsection{The connection between \texorpdfstring{$\om_M^\star$}{} and \texorpdfstring{$\om_m$}{}} 

With every positive sequence $M$ satisfying $M_0=1$ and $M_k^{1/k} \to \infty$
we associate a function $\om_M$ 
by setting 
\[
  \om_M (t) = - \log h_M(1/t) = \sup_{k \in \N}  \log \Big(\frac{t^k}{M_k}\Big), \quad t>0. 
\] 
Then $\om_M$ is increasing, convex in $\log t$, and zero for sufficiently small $t>0$. 
If $M$ is a weight sequence such that $\liminf_{k \to \infty} m_k^{1/k} >0$ and 
$\liminf_{k \to \infty} \mu_{Qk}/\mu_k >1$ for some $Q \in \N$, 
then $\om_M$ is a weight function.
See \cite{Komatsu73} and \cite[Lemma 12]{BMM07}. 
The proof of the latter shows that  
$\om_M(t) = o(t)$ as $t \to \infty$ provided that $m_k^{1/k} \to \infty$.

There is a connection between $\om_M^\star$ and $\om_m$. We found this in \cite[Lemma 5.7.8]{Debrouwere14}. 

\begin{lemma} \label{lem:star}
  Let $M$ be a weight sequence such that $m_k^{1/k} \to \infty$.
  Then 
  \begin{equation} \label{star}
    \A t > 0 : \om_M^\star(t) \le \om_m\Big(\frac1t\Big) \le \om_M^\star\Big(\frac {t}{e}\Big).
  \end{equation}
\end{lemma}

\begin{proof}
  We have $\om_M(t) = o(t)$ as $t \to \infty$ and so $\om_M^\star$ is well-defined. For $s>0$, 
\begin{align*}
\omega_M^{\star}(s):=\sup_{t\ge 0}\big(\omega_M(t)-st\big)
=\sup_{k\in\N}\sup_{t\ge 0}\Big(\log\Big(\frac{t^k}{M_k}\Big)-st\Big)=\sup_{k\in\NN}\log\Big(\frac{k^k}{(es)^kM_k}\Big),
\end{align*}
by an easy calculation. Using $k! \le k^k \le e^k k!$ we find 
\[
  \omega_M^{\star}(s) \le \sup_{k\in\N}\log\Big(\frac{1}{s^km_k}\Big) = \om_m\Big(\frac{1}{s}\Big) 
  \le \sup_{k\in\NN}\log\Big(\frac{k^k}{s^kM_k}\Big) = \om_M^\star\Big(\frac{s}{e}\Big)
\]
as required.
\end{proof}

\begin{corollary}
  Let $\om$ be a weight function satisfying $\om(t) = o(t)$ as $t \to \infty$.
  Let $\fW$ be the associated weight matrix.
  Then, for all $M \in \fW$ there exists $C\ge 1$ such that for all $t>0$ 
  \begin{equation} \label{eq:omMom1}
  \om^\star (t) \le C \om_M^\star\Big(\frac{t}{C}\Big) + C 
  \quad \text{ and } \quad \om_M^\star (t) \le C \om^\star\Big(\frac{t}{C}\Big) + C 
  \end{equation}
  as well as
  \begin{equation}
  \om^\star (t) \le C \om_m\Big(\frac{C}{t}\Big) + C
  \quad \text{ and } \quad
  \om_m(t) \le C \om^* \Big(\frac{1}{eC t}\Big) + C.
  \end{equation}
  In particular,
  \begin{equation} \label{eq:key1}
    \exp(\om^\star(t)) \le \Big(\frac{e}{h_m(t/C)}\Big)^C.
  \end{equation}
\end{corollary}

\begin{proof}
  By \cite[Lemma 5.7]{RainerSchindl12}, for each $M \in \fW$, we have $\om(t) = O(\om_M(t))$ and $\om_M(t) = O(\om(t))$ 
  as $t \to \infty$.	
  So \eqref{eq:omMom1} is a consequence of 
  \Cref{lem:omsistar}. 
  The rest follows from \Cref{lem:star}.
\end{proof}

In the proof of the following lemma 
log-convexity of the sequences was used.

\begin{lemma}[{\cite[Remark 2.5]{RainerSchindl16a}}] \label{lem:mg}
   Let $M$ and $N$ be weight sequences such that 
  \begin{equation}
    \E C \ge 1 \A k,j \in \N : M_{k+j} \le C^{k+j} N_j N_{k}. \label{mg0}
  \end{equation}
  Then $h_M(t) \le h_{N}(Ct)^2$ for all $t>0$.
\end{lemma}

We need a corresponding version for the sequences $m,n$ which are not log-convex in general. 
This can be achieved by using the connection between $\om_M^\star$ and $\om_m$.

\begin{lemma} \label{lem:hmodgrowth}
  Let $M$ and $N$ be weight sequences satisfying \eqref{mg0} and $m_k^{1/k} \to \infty$ and $n_k^{1/k} \to \infty$.
  Then there is a $D\ge 1$ such that $h_m(t) \le h_{n}(Dt)^2$ for all $t>0$.
\end{lemma}

\begin{proof}
  By \Cref{lem:mg}, $h_M(t) \le h_{N}(Ct)^2$ and hence $2\, \om_{N}(t) \le \om_M(Ct)$ for all $t>0$. 
  Then 
  \begin{align*}
    2\, \om^\star_{N}(t) = \sup_{s \ge 0} \big( 2\,\om_{N} (s) - 2ts\big)
    \le \sup_{s \ge 0} \big( \om_{M} (Cs) - 2ts\big) =  \om^\star_{M} \Big(\frac{2t}C\Big).
  \end{align*}
  By \eqref{star},
  \begin{align*}
    2\, \om_{n} \Big(\frac1t\Big) \le 2\,\om_{N}^\star\Big(\frac te\Big) 
    \le \om^\star_{M}\Big(\frac{2t}{eC}\Big) \le \om_m\Big(\frac{eC}{2t}\Big). 
  \end{align*}
  This entails the statement.
\end{proof}

\subsection{The heirs of a weight function}

We introduce notation for our convenience.

\begin{definition}[The heirs of a weight function]
Let $\om$ be a non-quasianalytic weight function. Then 
\begin{equation}\label{eq:scion}
\ka(t) = \ka_{\omega}(t):=\int_1^{\infty}\frac{\omega(tu)}{u^2}du=t\int_t^{\infty}\frac{\omega(u)}{u^2}du, \quad t>0,
\end{equation}
defines a weight function (possibly quasianalytic) satisfying $\ka(t) = o(t)$ as $t \to \infty$; 
cf.\ \cite[Remark 3.20]{BonetMeiseTaylor92}.   
Moreover, $\ka$ is concave; see \cite[Proposition 1.3]{MeiseTaylor88}. 
Since $\om$ is increasing we have $\ka\ge \om$, which implies $K^x \le W^x$ for all $x>0$, 
where $\{K^x\}_{x>0}$ is the weight matrix associated with $\ka$.

All weight functions $\si$ satisfying $\si(t) = o(t)$ and $\ka(t) = O(\si(t))$ as $t \to \infty$, i.e.,
\begin{equation} \label{eq:heir}
  \E C >0 \A t>0 : \int_1^{\infty}\frac{\omega(tu)}{u^2}du \le C \si(t) +C, 
\end{equation}
are called \emph{heirs of the weight function $\om$}.
A \emph{good heir of $\om$} is a heir of $\om$ which is a good weight function in the sense of \Cref{def:good}.
If $\om$ itself is a heir of $\om$, then $\om$ is said to be a \emph{strong} weight function.
\end{definition}

In particular, $\ka$ is a heir of $\om$. 
By \cite{BonetMeiseTaylor92}, 
the condition \eqref{eq:heir} is necessary and sufficient for the surjectivity of 
$j^\infty_{\{0\}}  : \cB^{\{\om\}}(\R^n) \to \cB^{\{\si\}}(\{0\})$. 
That a heir $\si$ satisfies $\si(t) = o(t)$ as $t \to \infty$ guarantees that we can work with the conjugate 
$\si^\star$.

\begin{lemma} \label{lem:heir}
  Let $\om$ be a non-quasianalytic weight function and $\si$ a heir of $\om$.
  Let $\fW = \{W^x\}_{x>0}$ and $\fS = \{S^x\}_{x>0}$ be the weight matrices associated with $\om$ and $\si$,
  respectively. 
  Then 
  \[
    \E C\ge 1 \A x >0 : S^{x} \le e^{1/x}\, W^{Cx}.
  \]
\end{lemma}

\begin{proof}
  By \eqref{eq:heir}, $\om \le \ka \le C \si + C$ and hence $\vh_\om \le C \vh_\si +C$.
  For the Young conjugates this means $\vh_\om^*(Ct) + C \ge C \vh_\si^*(t)$ which entails the assertion. 
\end{proof}

Next we recall that \eqref{eq:heir} can be equivalently stated with $\om$ replaced by its harmonic extension.
For a continuous function $u : \R \to \R$ with $\int_\R \frac{|u(t)|}{1+t^2}\, dt <\infty$, we define its 
\emph{harmonic extension} $P_u : \C \to \R$ by 
\[
  P_u(x+iy) := \begin{cases}
    \frac{|y|}{\pi} \int_\R \frac{u(t)}{(t-x)^2 + y^2}\, dt & \text{ if } y \ne 0,\\
    u(x) & \text{ if } y =0.
  \end{cases}
\]
Then $P_u$ is continuous on $\C$ and harmonic in the open upper and lower half plane.
If $\om$ is a weight function, we extend $\om$ to $\C$ by $z \mapsto \om(|z|)$, and $P_\om$ denotes the 
harmonic extension of $t \mapsto \om(|t|)$. We have $\om \le P_\om$, cf.\ \cite[Remark 1.6]{MeiseTaylor88}.

\begin{lemma}[{\cite[Lemma 3.3]{BonetMeiseTaylor92}}] \label{lem:scionharmonicextension}
	For a non-quasianalytic weight function $\om$ and a weight function $\si$ the following conditions are equivalent:
	\begin{enumerate}
	 	\item $\ka_\om(t) = O(\si(t))$ as $t \to \infty$.
	 	\item $P_\om(t) = O(\si(t))$ as $t \to \infty$. 
	 \end{enumerate} 
\end{lemma}

\subsection{Concave and good weight functions} \label{sec:good}

Let $\om$ be a non-quasianalytic weight function.
The weight function $\ka = \ka_\om$ defined in \eqref{eq:scion} is concave, and hence subadditive, 
since $\ka(0)= 0$. 
Since $\ka$ is the heir of $\om$ which defines the largest function space among all heirs of $\om$, it is of interest to  
find conditions which guarantee that $\ka$ is a good heir of $\om$.

Let us recall a result which relates concavity of a weight function with a condition on the associated 
weight matrix.

\begin{theorem}[{\cite[Theorem 3 and 5]{RainerSchindl14}}] \label{thm:subadditive}
Let $\om$ be a weight function and let $\fW = \{W^x\}_{x >0}$ be the associated weight matrix. 
Then the following conditions are equivalent:
\begin{enumerate}
  \item $\om$ is equivalent to its least concave majorant.
  \item $\exists C>0 \E t_0 >0 \A \la \ge 1 \A t \ge t_0 : \om(\la t)\le C \la \, \om(t)$.
  \item $\forall x>0 \E y>0   \E D \ge 1 \A  1 \le j \le k : (w^x_j)^{1/j} \le D\, (w^y_k)^{1/k}$. \label{eq:m6}
\end{enumerate}
\end{theorem} 

The equivalence of the first two conditions can be found in \cite{PetzscheVogt84} and is based on \cite[Lemma 1]{Peetre70}. 
The equivalence with the third condition was proved in \cite{RainerSchindl14} building on a result of \cite{FernandezGalbis06}, 
by showing that the conditions are all   
equivalent to several stability properties of the corresponding spaces of ultradifferentiable functions.

\begin{theorem} \label{thm:suffgood} 
  Let $\om$ be a weight function.
  Assume that the associated weight matrix $\fW= \{W^x\}_{x >0}$ satisfies
  \begin{equation} \label{eq:m5}
    \A x>0 \E y>0  : \vt_k^x \lesssim (W^y_k)^{1/k}. 
  \end{equation}
  Then $\om$ is a good weight function if and only if it is equivalent to its least concave majorant.
\end{theorem}

\begin{proof}
  By \eqref{eq:m5} and \eqref{mucompare}, for all $x>0$ there exists $y>0$ such that  
  $(W^x_k)^{1/k} \le \vt_k^x \lesssim (W^y_k)^{1/k} \le \vt_k^y$ and consequently,
  $$(w^x_k)^{1/k} \lesssim \frac{\vt_k^x}{k} \lesssim (w^y_k)^{1/k} \lesssim \frac{\vt_k^y}{k}.$$
  Then clearly the conditions \eqref{eq:good} and \ref{thm:subadditive}\eqref{eq:m6} are equivalent. 
\end{proof}

\begin{remark}
  Note that \eqref{eq:m5} is not invariant under equivalence of weight functions, in contrast to the three equivalent conditions 
  in \Cref{thm:subadditive}; compare with \Cref{rem:good}(2).
\end{remark}

\begin{corollary}
  Let $\om$ be a non-quasianalytic weight function.
  Then $\ka_\om$, defined in \eqref{eq:scion}, is a good heir of $\om$ provided that its associated weight matrix satisfies 
  \eqref{eq:m5}.
\end{corollary}

This raises the following question.

\begin{question} \label{q:concave}
  Is every concave weight function equivalent to a good one?   
\end{question}

A strong weight function $\om$ is equivalent to the concave weight function $\ka_\om$. We will discuss the 
relation between strong and good weight functions in \Cref{sec:strong}. 

\begin{remark}
For the sake of completeness we remark that \eqref{eq:m5} amounts to the following condition on the secants of $\vh^*$:
\begin{equation*}
	\A x>0 \E y >0 \E C>0 \A k \in \N_{>0} : 
	\frac{\vh^*(xk) - \vh^*(xk-x)}{x} \le \frac{\vh^*(yk)}{yk} + C.	 	
\end{equation*}	
A weight function $\om$ is good if and only if
\begin{align*}
	\A x>0 \E y >0 &\E C>0 \A 1 \le j \le k : 
	\notag \\
	\log k - \log j &\le \frac{\vh^*(yk) - \vh^*(yk-y)}{y} - \frac{\vh^*(xj) - \vh^*(xj-x)}{x} + C. 	 	
\end{align*}
\end{remark}

\section{A convenient partition of unity} \label{sec:partition}

In this section we construct a special partition of unity which will be a cornerstone for the extension theorem.
The construction is based on a result of \cite{BBMT91}.

\subsection{Special bump functions} 

The following proposition is due to \cite{BBMT91} in the case that $\om$ is a strong concave weight function and 
$\si = \om$. The proof of the general case (with $\si \ne \om$) requires some slight modifications of the original proof 
of \cite{BBMT91}. We recall the main steps and detail the passages,  
where a transition from $\om$ to $\si$ occurs.

In this section $\fW = \{W^x\}_{x>0}$ will always be the weight matrix associated with the weight function $\om$. 

\begin{proposition}\label{proposition22BBMTWhitney}
Let $\om$ be a non-quasianalytic concave weight function and let $\si$ be a heir of $\om$. 
Then for each $n\in \N_{>0}$ there exist $m\in\N_{>0}$, $M>0$, and $0<r_0<1/2$ 
such that for all $0<r<r_0$ there are functions $f_{n,r}\in C^\infty(\RR)$ satisfying the following properties:
\begin{gather}\label{proposition22BBMTWhitneyequ}
0\le f_{n,r}\le 1,\quad 
\supp f_{n,r} \subseteq \big[-\tfrac{9}{8}r,\tfrac{9}{8}r \big],\quad 
f_{n,r}|_{[-r,r]} =1, 
\\
\label{proposition22BBMTWhitneyequ1}
\sup_{x\in\RR, \, j \in \N} \frac{|f^{(j)}_{n,r}(x)|}{W^m_j}\le M\exp\Big(\frac{1}{n}\si^{\star}(nr)\Big).
\end{gather}
The proof will show that $m=cn$ for some $c\in\NN_{>0}$ independent of $n$.
\end{proposition}

Note that, in \Cref{proposition22BBMTWhitney}, $\si$ need not be a \emph{good} heir of $\om$.

The following two lemmas can be taken without modification from \cite{BBMT91}. 
They are based on H\"ormander's $L^2$-method to construct entire functions and on a Paley--Wiener theorem.

\begin{lemma}[{\cite[Lemma 2.3]{BBMT91}}]\label{lemma23BBMTWhitney}
Let $\omega$ be a non-quasianalytic weight function. Then there exists $A>0$ such that for each $0<r\le 1$, each $k\in\NN$, 
and each subharmonic function $u$ on $\C$ satisfying
\begin{equation*} 
u(z)\le r|\Im(z)|-\frac{\omega(z)}{k} \quad \text{ for all } z \in \C, 
\end{equation*}
there exists an entire function $F$ on $\C$ with $F(0)=1$ and
\begin{equation*} 
|F(z)|\le A\exp\Big(r|\Im(z)|-\frac{\omega(z)}{k}+3\log(1+|z|^2)\Big)\sup_{|w|\le 1}\exp(-u(w))
\end{equation*}
for all $z \in \C$.
\end{lemma}

\begin{lemma}[{\cite[Lemma 2.4]{BBMT91}}]\label{lemma24BBMTWhitney}
Let $\omega$ be a non-quasianalytic weight function.  
There exists $L\in\NN_{>0}$ such that for each $k\in\NN_{>0}$ there exists 
$B>0$ such that for all $0<r<1/2$ the following holds. 
If there is an entire 
function $F$ with $F(0)=1$ such that 
\begin{equation*}
\E M>0 \A z\in\C : |F(z)|\le M\exp\Big(r|\Im(z)|-\frac{\omega(z)}{k}\Big),
\end{equation*}
then there exists $\psi \in \cB^{\{\omega\}}(\RR)$ with the following properties:
\begin{gather*}
0\le\psi\le 1,\quad \psi(x)=0 \text{ for } x\le -r,\quad \psi(x)=1 \text{ for } x\ge r,
\\
\sup_{x\in\RR,\, j\in\NN}\frac{|\psi^{(j)}(x)|}{W^{2Lk}_j}\le BM^2.
\end{gather*}
\end{lemma}

Next we generalize \cite[Lemma 2.5]{BBMT91}. 
For $T>1$ we define $\om_T:\RR\rightarrow[0,\infty)$ by
\begin{equation}\label{lemma25BBMTWhitneyequ}
  \om_T(t) := \begin{cases}
    \om(t) & \text{ if } |t| \ge T, \\
    \frac{\om'(T)}{2T}t^2-\frac{\om'(T)}{2}T+\om(T) &  \text{ if } |t| \le T. 
  \end{cases}
\end{equation}

\begin{lemma}\label{lemma25BBMTWhitney}
Let $\omega$ be a non-quasianalytic concave weight function and let $\si$ be a heir of $\om$.
Then there is a $D>0$ such that for all $T > 1$,
\begin{equation}\label{lemma25BBMTWhitneyequ1}
 \sup_{x\in\R}\frac{\partial}{\partial y}P_{\om_T}(x+i)\le D\, \frac{\si(T)}{T}.
\end{equation}
\end{lemma}

\begin{proof}
It suffices to follow the proof of \cite[Lemma 2.5]{BBMT91} and replace the use of the estimate 
$\int_1^\infty \frac{\om(tu)}{u^2}\, du  \le C \om(t) + C$ 
(that is \cite[1.7(1)]{BBMT91}) by the estimate \eqref{eq:heir}.
\end{proof}

Let $\omega, \om_T$ be as in \Cref{lemma25BBMTWhitney} and \eqref{lemma25BBMTWhitneyequ}. 
We consider $h_T:\C\rightarrow\RR$ given by
\begin{equation*} 
h_T(z):= \begin{cases}
  P_{\om_T}(z+i) & \text{ if } \Im(z) \ge 0,\\
  P_{\om_T}(z-i) & \text{ if } \Im(z) < 0.
\end{cases}
\end{equation*}
If $\om_T$ is replaced by $\om$, then we will write $h$ for the corresponding function.
By the symmetry of $\om_T$ and of the Poisson kernel,
$h_T$ is continuous on $\C$. We have 
\begin{equation} \label{eq:hTh}
	h_T(z) - \om (T) \le h(z) \le h_T(z) \quad \text{ for all } z \in \C, ~T>1, 
\end{equation}
by \cite[2.7(2)]{BBMT91}. 
The following generalizes \cite[Lemma 2.7]{BBMT91}.

\begin{lemma}\label{lemma27BBMTWhitney}
Let $\omega$ be a non-quasianalytic concave weight function and let $\si$ be a heir of $\om$.
Then there exist $E,F,G>0$ such that for all $T>1$ and all $z \in \C$, 
\begin{equation}\label{lemma27BBMTWhitneyequ}
E^{-1}h_T(z)-F\om(T)\le\si(z),\quad \om(z)\le h_T(z)+G.
\end{equation}
\end{lemma}

\begin{proof}
The proof of \cite[Lemma 2.7]{BBMT91} yields that there exists $G>0$ such that 
\[
	|P_\om(z+w) - P_\om(z)| \le G \quad \text{ for all } z,w \in \C, ~ |w|\le 1.
\]
Together with \eqref{eq:hTh} this implies
\[
	\om(z) \le P_\om(z) \le P_\om(z+i) + G = h(z) + G \le h_T(z) + G,
\]
if $\Im (z) \ge 0$, 
and similarly for $\Im (z) <0$. This gives the second inequality in \eqref{lemma27BBMTWhitneyequ}.

For the first inequality note that $P_\om \le C \si + C$, by \Cref{lem:scionharmonicextension}.
Then, by \eqref{eq:hTh},  
\[ 
h_T(z) - \om(T) - G \le h(z) - G \le P_\om(z) \le C \si(z) + C, 
\]
which easily implies the first inequality in \eqref{lemma27BBMTWhitneyequ}. 
\end{proof}

Now we generalize \cite[Lemma 2.8]{BBMT91}.

\begin{lemma}\label{lemma28BBMTWhitney}
Let $\omega$ be a non-quasianalytic concave weight function and let $\si$ be a heir of $\om$.
Then for each $n\in\NN_{>0}$ there exist $m\in\NN_{>0}$, $M>0$ and $0<r_0<1/2$ such that for all $0<r<r_0$ 
there are functions  
$g_{n,r}\in C^\infty(\R)$ satisfying the following properties:
\begin{gather}\label{lemma28BBMTWhitneyequ}
0\le g_{n,r}\le 1,\quad g_{n,r}=0 \text{ for } x\le -r,\quad g_{n,r}(x)=1 \text{ for } x\ge r,
\\
\label{lemma28BBMTWhitneyequ1}
\sup_{x\in\R,\, j\in\NN}\frac{|g_{n,r}^{(j)}(x)|}{W^m_j}\le M\exp\Big(\frac{1}{n}\si^{\star}(n r)\Big).
\end{gather}
\end{lemma}

\begin{proof}
There is a constant $C$ such that $\om \le C \si +C$.
Let $A,L,D,E,F,G$ be the constants arising in \Cref{lemma23BBMTWhitney}, \Cref{lemma24BBMTWhitney}, 
\Cref{lemma25BBMTWhitney}, and \Cref{lemma27BBMTWhitney}. We can assume that $C,L,D,E,F$ are positive integers. 
For $n\in\N_{>0}$ let 
\begin{equation} \label{eq:choiceofk}
	k:=(2CEF+D)n  \quad \text{  and } \quad m:=4Lk.
\end{equation}
Choose $0<r_0<1/2$ such that the equation $\si(t)/t=r_0 k/D$ has a solution $t>1$. 
Fix $0<r<r_0$ and choose $T = T(k,r)>1$ such that
\begin{equation}\label{lemma28BBMTWhitneyequ2}
\si(T)=T\,\frac{rk}{D}.
\end{equation}
Define $u_{n,r}:\C\rightarrow\RR$ by
\[
u_{n,r}(z):=r |\Im(z)|-\frac{h_T(z)}{k}-\frac{G}{k}.
\]
Then, by \eqref{lemma27BBMTWhitneyequ}, for all $z \in \C$
\begin{align}\label{lemma28BBMTWhitneyequ3}
u_{n,r}(z)&\le r|\Im(z)|-\frac{\om(z)}{k}, \\ 
-u_{n,r}(z)&\le-r|\Im(z)|+\frac{E}{k}\si(z)+\frac{EF}{k}\om(T)+\frac{G}{k}.\label{eq:BBMT2}
\end{align}
By definition $u_{n,r}$ is subharmonic on the open upper and lower half plane. 
By \eqref{lemma25BBMTWhitneyequ1} and \eqref{lemma28BBMTWhitneyequ2}, we have
\[
-\frac{1}{k}\frac{\partial}{\partial y}h_T(x)
=-\frac{1}{k}\frac{\partial}{\partial y}P_{\om_T}(x+i)\ge-\frac{D}{k}\frac{\si(T)}{T}
=-r,
\]
for all $x\in\RR$. 
Thus, for each non-negative $g\in C^\infty_c(\C)$, 
\[
\int_{\C}u_{n,r}(z)\Delta g(z)\, d\lambda(z)
=2\int_{-\infty}^{\infty}\Big(r-\frac{1}{k}\frac{\partial}{\partial y}h_T(x)\Big)g(x)\, dx\ge 0,
\]
whence $u_{n,r}$ is subharmonic on $\C$. 
By \eqref{lemma28BBMTWhitneyequ3} and \Cref{lemma23BBMTWhitney}, 
there is an entire function $F_{n,r}$ with $F_{n,r}(0)=1$ and 
\begin{equation*}\label{lemma28BBMTWhitneyequ4}
|F_{n,r}(z)|\le A\exp\Big(r|\Im(z)|-\frac{\om(z)}{k}+3\log(1+|z|^2)\Big)\sup_{|w|\le 1}\exp(-u_{n,r}(w))
\end{equation*}
for all $z \in \C$.
By \eqref{eq:BBMT2} and since $\om \le C \si +C$, there is a constant $K(n)>0$ such that 
\begin{align*}
\sup_{|w|\le 1}\exp(-u_{n,r}(w)) &\le K(n) \exp\Big(\frac{CEF}{k}\si(T)\Big).
\end{align*}
Using $\log (t) = o(\om(t))$ as $t \to \infty$ (i.e., \eqref{om3}), we find that (for a possibly larger constant $K(n)$)
\begin{align*}
|F_{n,r}(z)| &\le K(n)\exp\Big(r|\Im(z)|-\frac{\om(z)}{2k}\Big)\exp\Big(\frac{CEF}{k}\si(T)\Big)
\end{align*}
for all $z \in \C$.
By \Cref{lemma24BBMTWhitney}, there is a constant $B(n)>0$ and functions $g_{n,r}\in C^\infty(\R)$ 
satisfying \eqref{lemma28BBMTWhitneyequ} and
\begin{equation*} 
\sup_{x \in \R,\,j\in\NN}\frac{|g_{n,r}^{(j)}(x)|}{W^{4Lk}_j}\le B(n)K(n)^2\exp\Big(\frac{2CEF}{k}\si(T)\Big).
\end{equation*}
By the definition of $\si^{\star}$, \eqref{lemma28BBMTWhitneyequ2}, and the choice of $k$ (see \eqref{eq:choiceofk}), 
\begin{align*}
	\frac{1}{n} \si^{\star}(n r) &\ge \frac{1}{n} \big(\si(T)-n rT\big) 
=\si(T)\Big(\frac{1}{n}-\frac{D}{k}\Big)
=\si(T)\frac{2CEF}{k}	
\end{align*}
which implies \eqref{lemma28BBMTWhitneyequ1}. 
The proof is complete.
\end{proof}

\begin{proof}[Proof of \Cref{proposition22BBMTWhitney}]
Follow the proof of \cite[Proposition 2.2 (p.168)]{BBMT91} and use \Cref{lemma28BBMTWhitney}.
\end{proof}

\subsection{A special partition of unity}

Let $E \subseteq \R^n$ be a compact set.
We denote by $d(Q,E)$ the Euclidean distance of a closed set $Q \subseteq \R^n$ to $E$, 
in particular, $d(x,E) = \inf\{|x-y| : y \in E\}$.

\begin{lemma}[{\cite[p.167]{Stein70}, \cite[Lemma 3.2]{Bruna80}, \cite[Lemma 3.6]{BBMT91}}] \label{Whitneycubes}
  Let $E \subseteq \R^n$ be a non-empty compact set. 
  There exists a collection of closed cubes $\{Q_i\}_{i \in \N}$ with sides parallel to the axes 
  satisfying the following properties:
  \begin{enumerate}
    \item $\R^n \setminus E = \bigcup_{i \in \N} Q_i$.
    \item The interiors of the $Q_i$ are pairwise disjoint.
    \item $\on{diam} Q_i \le d(Q_i,E) \le 4 \on{diam} Q_i$ for all $i \in \N$. 
    \item Let $Q_i^*$ be the closed cube which has the same center as $Q_i$ expanded by the factor $9/8$.
    For each $i \in \N$ the number of cubes $Q_j^*$ which intersect $Q_i^*$ is bounded by $12^{2n}$.
    \item There exist $b_1, B_1>0$ (independent of $E$) such that for all $i,j \in \N$ with $Q_i^* \cap Q_j^* \ne \emptyset$ 
    we have 
      $b_1  \on{diam} Q_i \le \on{diam} Q_j \le B_1 \on{diam} Q_i$. 
    
  \end{enumerate}
\end{lemma}

For every $x \in \R^n$ we denote by $\hat x$ any point in $E$ with $|x-\hat x| = d(x,E)$.

\begin{corollary}
  \label{cor:Whitneycubes}
  In the setting of \Cref{Whitneycubes}, let $x_i$ be the center of $Q_i$. 
  Then for all $x \in Q^*_i$,
  \begin{gather*}
    \frac{1}{2} d(x,E) \le  d(x_i,E) \le 3 d(x,E),
    \\
    \frac{1}{3} \on{diam} Q_i \le d(x,E) \le 9 \on{diam} Q_i,
  \end{gather*}
  \[
    |\hat x_i-x|\le 2 d(x_i,E), \quad |\hat x_i-\hat x| \le 4 d(x_i,E).
  \]
\end{corollary}
 
 \begin{proof}
 All this follows easily from  
 \[
  |x_i-x| \le \frac{9}{8} \frac{\on{diam} Q_i}{2} \le \frac{9}{16} d(Q_i,E) \le \frac{9}{16} d(x_i,E). \qedhere
 \]
 \end{proof}

In analogy with \cite[Lemma 3.7]{BBMT91} we may conclude the following.

\begin{proposition}\label{Proposition6}
Let $E \subseteq \R^n$ be a non-empty compact set
and let $\{Q_i\}_{i \in \N}$ be the family of cubes provided by \Cref{Whitneycubes}.
Let $\om$ be a non-quasianalytic concave weight function and let $\si$ be a heir of $\om$.  
Then for all $p\in\NN_{>0}$ there exist $m\in\N_{>0}$, $M>0$, $0 < r_0 < 1/2$,  and a family of smooth functions 
$\{\varphi_{i,p}\}_{i\in \N}$ satisfying
\begin{enumerate}
\item $0\le\varphi_{i,p}\le 1$ for all $i\in\NN$,
\item $\supp\varphi_{i,p}\subseteq Q_i^*$ for all $i\in\NN$,
\item $\sum_{i\in\N}\varphi_{i,p}(x)=1$ for all $x\in\R^n\setminus E$,
\item if $d(Q_i,E) \le r_0/B_1$, then for all $\be\in\N^n$ and $x\in\R^n\setminus E$,
\[
  |\varphi^{(\be)}_{i,p}(x)|\le M W^m_{|\be|} 
  \exp\Big(\frac{A_1(n)}{p}\si^{\star}\Big(\frac{b_1 p}{A_2(n)} \on{diam} Q_i\Big)\Big),
\]
for constants $A_1(n)\le A_2(n)$ only depending on $n$.
\end{enumerate}
\end{proposition}

\begin{proof}
  Let $p$ be a positive integer.      
   Let $f_{p,r}$, for $0< r<r_0=r_0(p)$, be the functions provided by \Cref{proposition22BBMTWhitney}. 
   The function
   \[
      g_{p,r}(x) := f_{p,r}(x_1) \cdots f_{p,r}(x_n), \quad x = (x_1,\ldots,x_n) \in \R^n,    
   \]
   satisfies $0 \le g_{p,r} \le 1$, has support in the cube centered at $0$ with sidelength $9r/4$ and 
   equals $1$ in the cube centered at $0$ with sidelength $2r$.
  There exist $m$, $M$ such that for all $r<r_0$,
   $\be\in \N^n$, and $x \in \R^n$
  \begin{align} \label{eq:part1}
      |g^{(\be)}_{p,r}(x)|\le M W^m_{|\be|}\exp\Big(\frac{n}{p}\si^{\star}(pr)\Big),
  \end{align}    
  thanks to \Cref{lem:basicM}(2) and \Cref{lemma4}(5). 
  
  Let $2r_i$ denote the sidelength of $Q_i$ and $x_i$ its center.   
  If $r_i < r_0$, or equivalently, $\on{diam} Q_i < 2 \sqrt n r_0$, 
  then we define
  \[
  \psi_{i,p}(x):=g_{p,r_i}(x-x_i).
  \]
    Then 
    \begin{equation} \label{eq:partition}
      0 \le \ps_{i,p} \le 1,\quad \supp \ps_{i,p} \subseteq  Q_i^*, 
      \quad  \ps_{i,p}|_{Q_i} = 1.  
    \end{equation}
    Moreover, by \eqref{eq:part1},
  \begin{equation} \label{eq:part11}
  	|\psi^{(\be)}_{i,p}(x)| 
      \le M W^m_{|\be|}\exp\Big(\frac{n}{p}\si^{\star}\Big(\frac{p}{2 \sqrt n} \on{diam} Q_i\Big)\Big).
  \end{equation}
  For those $i$ with $r_i \ge r_0$, we just choose arbitrary $C^\infty$-functions $\ps_{i,p}$ satisfying 
     \eqref{eq:partition}. 

     Then put
  \[
    \varphi_{1,p}:=\psi_{1,p},\quad \varphi_{i,p}:=\psi_{i,p}\prod_{k=1}^{i-1}(1-\psi_{k,p}),~i\ge 2.
  \]
  It is easy to check that (1)--(3) are satisfied (cf.\ \cite[Lemma 3.3]{Bruna80}).

  Assume that $d(Q_i,E) \le r_0/B_1$. 
  Then \Cref{Whitneycubes}(3)\&(5) guarantees that the diameters of the cubes 
  which correspond to nontrival factors in the product which defines $\vh_{i,p}$ satisfy 
  $\on{diam} Q_k \le r_0 < 2 \sqrt n r_0$. So for those factors we have the estimate \eqref{eq:part11}. 
  There are at most $12^{2n}$ such factors.
  Consequently, by \Cref{lem:basicM}(2) and \Cref{Whitneycubes}(5),  
  we get
  \[
  |\varphi^{(\be)}_{i,p}(x)|\le M^{12^{2n}}
  W^m_{|\be|}\exp\Big(\frac{n 12^{2n}}{p}\si^{\star}\Big(\frac{b_1 p}{2 \sqrt n} \on{diam} Q_i\Big)\Big). 
  \]
  This implies (4), since $\si^\star$ is decreasing.
\end{proof}

\section{The extension theorem} \label{sec:extension}

In this section we prove the implication $(2) \Rightarrow (1)$ in \Cref{main}.
We subdivide the proof into three parts for two reasons:
\begin{enumerate}
  \item The proof is (by nature) quite technical. We hope that the subdivison 
  improves the clarity of the presentation.
  \item The organization into parts should make it easier to see, where in the 
  line of arguments the particular assumptions are needed.
  The first two parts \Cref{sec:prelimI} and \Cref{sec:prelimII} prepare the 
  stage with preliminary lemmas and estimates. This is the place, where we 
  use that the heir $\si$ of $\om$ is \emph{good}. The actual proof of the 
  extension theorem is given in the third part, i.e., \Cref{sec:extproof}. 
\end{enumerate}
In \Cref{sec:CC} we deduce a consequence for Denjoy--Carleman classes and compare 
it with the result of \cite{ChaumatChollet94}. 
Finally, in \Cref{sec:strong} we discuss the relation of strong and good weight 
functions.

\subsection{Preliminaries, I} \label{sec:prelimI}

Let $E \subseteq \R^n$ be a compact set.
Let $S=(S_k)$ be a weight sequence satisfying $s_k^{1/k} \to \infty$ and let
$F= (F^\al)_{\al}$ be a Whitney ultrajet of class $\cB^{\{S\}}$ on $E$, i.e.,  
there exist $C>0$ and $\rh \ge 1$ 
such that 
\begin{gather}
  |F^\al(a)| \le C \rh^{|\al|} \,  S_{|\al|}, \quad \al \in \N^n,~ a \in E,
   \label{jets1}
  \\
  |(R^k_a F)^\al(b)| \le C \rh^{k+1} \, |\al|!\, s_{k+1}\,  |b-a|^{k+1-|\al|}, \quad k \in \N,\, |\al| \le k,~ a,b \in E.  
  \label{jets2}
\end{gather} 
The extension of $F$ will be of the form
\begin{equation} \label{extensionformula}
  \sum_{i\in \N} \vh_{i,p}(x) \, T_{\hat x_i}^{2 \ol \Ga_{s'}(L d(x_i)) } F(x),   \quad  x \in \R^n \setminus E,
\end{equation}
where 
\begin{itemize}
  \item $\{\vh_{i,p}\}_{i \in \N}$ is a partition of unity provided by \Cref{Proposition6} 
        ($x_i$ is the center of the cube $Q_i$), 
  \item $S'$ is a suitable weight sequence and $L$ is a constant, both depending on $S$. 
\end{itemize}
For simplicity of notation we use the abbreviation $d(x) := d(x,E)$.
Recall that $\hat x$ denotes any point in $E$ with $d(x) = |x-\hat x|$.

We begin with several estimates for the Taylor polynomials appearing in \eqref{extensionformula}.

\begin{lemma} \label{proposition10}
  For $a_1,a_2 \in E$, $x \in \R^n$ and $|\al| \le q$,
  \begin{equation*} 
    |(T^q_{a_1}F-T^q_{a_2}F)^{(\al)}(x)|\le C (2n^2 \rh)^{q+1}|\al|! \, s_{q+1}(|a_1-x|+|a_1-a_2|)^{q+1-|\al|}.
  \end{equation*}
\end{lemma}

\begin{proof}
  This is straightforward; for details see 
\cite[Proposition 10]{ChaumatChollet94}.
\end{proof}

\begin{lemma}\label{proposition9}
Let $S$, $S'$ be weight sequences satisfying $s_k^{1/k} \to \infty$, $(s'_k)^{1/k} \to \infty$, and  
\begin{equation} \label{eq:mass1}
  \E \la \le 1 \A t>0 : 2  \ol \Ga_{s'}(t) \le \ul \Ga_s(\la t).
\end{equation}
Then there is a constant $D_1 = D_1(S,S') >1$ such that,
for all Whitney ultrajets $F= (F^\al)_{\al}$ of class $\cB^{\{S\}}$ that satisfy \eqref{jets1} and \eqref{jets2},
all $L \ge D_1 \rh$, all $x \in \R^n$, and $\al\in \N^n$, 
\begin{align}
  |(T_{\hat x}^{2  \ol \Ga_{s'} (L d(x))} F)^{(\al)}(x)| &\le C (2L)^{|\al|+1} S_{|\al|},  \label{prop91} 
  \intertext{and, if $|\al| < 2 \ol \Ga_{s'}(L d(x))$,}
  |(T_{\hat x}^{2  \ol \Ga_{s'}(L d(x))}F)^{(\al)}(x)-F^\al(\hat x)| &\le C (2L)^{|\al|+1} |\al|!\, s_{|\al|+1} d(x).
  \label{prop92}
\end{align}
\end{lemma}

\begin{proof}
For \eqref{prop91} we may restrict to the case $|\al| \le 2  \ol \Ga_{s'}(L d(x))$. By \eqref{jets1},
\begin{align}
  |(T_{\hat x}^{2 \ol \Ga_{s'}(Ld(x))}F)^{(\al)}(x)|
  &\le\sum_{\substack{\al \le \be\\ |\be|\le2  \ol \Ga_{s'}(Ld(x))}} 
  \frac{|x-\hat{x}|^{|\be|-|\al|}}{(\be-\al)!} C \rh^{|\be|} S_{|\be|}
  \notag \\
  &\le C |\al|! \sum_{\substack{\al \le \be\\ |\be|\le2  \ol \Ga_{s'}(Ld(x))}}
  \frac{|\be|!\,(n d(x))^{|\be|-|\al|}}{|\al|!\, (|\be|-|\al|)!}  \rh^{|\be|} s_{|\be|}
  \notag \\
  &\le  \frac{C |\al|! }{(n d(x))^{|\al|}} \sum_{\substack{\al \le \be\\ |\be|\le2 \ol \Ga_{s'}(Ld(x))}}
  (2n \rh  d(x))^{|\be|} s_{|\be|}
  \notag \\
  &\le \frac{C|\al|! }{(n d(x))^{|\al|}} \sum_{j = |\al|}^{2 \ol \Ga_{s'}(Ld(x))}
  (2n^2 \rh  d(x))^{j} s_{j},  \label{calculation} 
\end{align}
since the number of $\be \in \N^n$ with $|\be|=j$ is bounded by $n^j$.
By \eqref{eq:mass1}, we may let $j$ run from $|\al|$ to $\ul \Ga_s(L\la d(x))$ in the sum on the right-hand side of 
\eqref{calculation}. For such $j$ we have $(L \la d(x))^{j} s_{j}\le (L \la d(x))^{|\al|} s_{|\al|}$, 
by \Cref{basic}\eqref{eq:ulGa3}, 
and hence
\begin{align*}
  |(T_{\hat x}^{2  \ol \Ga_{s'}(Ld(x))}F)^{(\al)}(x)|
  &\le C S_{|\al|} \Big(\frac{L \la}{n}\Big)^{|\al|} \sum_{j = |\al|}^{\ul \Ga_s(L\la d(x))}
  \Big(\frac{2n^2 \rh}{L \la } \Big)^{j}.  
\end{align*}
We obtain \eqref{prop91} if $L$ is chosen such that $2n^2\rh /(L\la)\le 1/2$; 
then $D_1= 4n^2/\la$.

For \eqref{prop92} note that, if $|\al| < 2 \ol \Ga_{s'}(L d(x))$, then 
\[
  (T_{\hat x}^{2 \ol  \Ga_{s'}(L d(x))}F)^{(\al)}(x)-F^\al(\hat x) 
  = \sum_{\substack{\al \le \be\\ |\al| < |\be|\le2  \ol \Ga_{s'}(Ld(x))}}\frac{(x-\hat{x})^{\be-\al}}{(\be-\al)!} F^\be(\hat x).
\]
Thus the same arguments yield \eqref{prop92}.
\end{proof}

\subsection{Preliminaries, II} \label{sec:prelimII}

Let $\si$ be a good weight function and let $\fS$ be the associated weight matrix.
Let $E$ be a compact subset of $\R^n$ and $F \in \cB^{\{\si\}}(E)$.
There exist $S \in \fS$, $C>0$ and $\rh \ge 1$ 
such that \eqref{jets1} and \eqref{jets2} hold.
By \Cref{cor:essence}, there are $\dot S, \ddot S, \dddot S \in \fS$ satisfying $S \le \dot S \le \ddot S \le \dddot S$
such that: 
\begin{align} \label{eq:ext1} 
     \E D = D(S)\ge 1 &\A t >0 : 
     \notag \\
    \ol \Ga_{\dddot s}(D^3t) &\le \ul \Ga_{\ddot s}(D^2t) 
    \le \ol \Ga_{\ddot s}(D^2t) 
    \le \ul \Ga_{\dot s}(Dt) \le \frac{\ul \Ga_{s}(t)}{2},
    \intertext{as well as} 
     \label{eq:mge1}
     &\A j,k \in \N : s_{j+k} \le \dot s_j \dot s_k, 
    \\ \label{eq:mge2}
    &\A j,k \in \N : \dot s_{j+k} \le \ddot s_j \ddot s_k.
  \end{align}   
Let $\{Q_i\}_{i\in \N}$ be the family of cubes provided by \Cref{Whitneycubes} and let $b_1,B_1$ be the constants from 
  \Cref{Whitneycubes}. Let $x_i$ be the center of $Q_i$.

\begin{lemma} \label{lem:H1}
  There is a constant $C_1 = C_1(S)>0$ such that for all $L>C_1 \rh$, all $\be \in \N^n$, 
  and all $x \in Q_i^*$ with $d(x)<1$, 
  \begin{align} \label{H1}
  |\p^\be (T_{\hat x_i}^{2 \ol \Ga_{\ddot s}(L d(x_i)) } F - T_{\hat x}^{2 \ol \Ga_{\ddot s}(L d(x_i)) } F) (x)| 
  &\le C   
    L^{|\be|+1}  \ddot S_{|\be|} \, h_{\ddot s}(L d(x_i)).
\end{align}
\end{lemma}

\begin{proof}
  It suffices to consider $|\be| \le 2 \ol \Ga_{\ddot s}(L d(x_i)) =: q$. Let $H_1$ denote the left-hand side of \eqref{H1}.
By \Cref{proposition10}, 
\begin{align*}
  H_1 
  \le C (2n^2 \rh)^{q+1} |\be|! \, 
  s_{q+1}  (|\hat x_i-x|+|\hat x_i-\hat x|)^{q+1-|\be|}.
\end{align*}
By \eqref{eq:mge1} and \eqref{eq:mge2}, 
\[
  s_{q+1} \le \dot s_1 \dot s_{q}   =  \dot s_1 \dot s_{2 \ol \Ga_{\ddot s}(L d(x_i))}  \le 
  \dot s_1 \ddot s_{\ol \Ga_{\ddot s}(L d(x_i))}^2.
\]
Together with
\Cref{cor:Whitneycubes}, we conclude that     
\begin{align*}
  H_1 
  &\le C \dot s_1 (2n^2  \rh)^{q+1} 
  |\be|! \,  \ddot s_{\ol \Ga_{\ddot s}(L d(x_i))}^2 (6 d(x_i))^{q+1-|\be|}.
\end{align*} 
By the definition of $\ol \Ga_{\ddot s}$,
\[
h_{\ddot s}(L d(x_i)) = \ddot s_{\ol \Ga_{\ddot s}(L d(x_i))} (L d(x_i))^{\ol \Ga_{\ddot s}(L d(x_i))} \le 
\ddot s_{|\be|} (L d(x_i))^{|\be|} \quad \text{ for all } \be.
\] 
Since $d(x_i) \le 3\, d(x)$, by \Cref{cor:Whitneycubes}, we find  
\begin{align*}
  H_1 
  &\le 36 \, C  \dot s_1  n^2    \rh \Big(\frac{12n^2  \rh}{L}\Big)^{q} 
   d(x)\, L^{|\be|}\, |\be|!\,\ddot s_{|\be|} \, h_{\ddot s}(L d(x_i)).
\end{align*}
If $L >  36 n^2   \dot s_1 \, \rh$ and $d(x)<1$, then \eqref{H1} follows. 
\end{proof}

\begin{lemma} \label{lem:H2}
	There is a constant $C_2 = C_2(S)>0$ such that for all $L>C_2 \rh$, all $\be \in \N^n$, and all $x \in Q_i^*$ 
	with $d(x)<1$, 
	\begin{align} \label{H2}
  |\p^\be (T_{\hat x}^{2 \ol \Ga_{\ddot s}(L d(x_i)) } F - T_{\hat x}^{2 \ol \Ga_{\ddot s}(L d(x)) } F) (x)| 
  &\le  C \Big(\frac{3L D }{n}\Big)^{|\be|+1} \dddot S_{|\be|}  h_{\dddot s}(3 LD d(x)).   
\end{align}  
\end{lemma}

\begin{proof}
Let $H_2$ denote the left-hand side of \eqref{H2}.
	Using \eqref{eq:ext1}, \Cref{cor:Whitneycubes} and the fact that $\ul \Ga_s$ is decreasing, it is easy to see that 
both $2 \ol \Ga_{\ddot s}(L d(x_i))$ and $2 \ol \Ga_{\ddot s}(L d(x))$ are majorized by $\ul \Ga_s(L \la  d(x))$ 
for some $\la <1$.
So the degree of the polynomial $T_{\hat x}^{2 \ol \Ga_{\ddot s}(L d(x_i)) } F - T_{\hat x}^{2 \ol \Ga_{\ddot s}(L d(x)) } F$
is at most $\ul \Ga_s(L \la  d(x))$. 
Similarly the valuation of the polynomial is at least 
$2 \ol \Ga_{\dddot s}(3 L D d( x)) =: 2 q$,
indeed, using that $\ul \Ga_{\ddot s}$ is decreasing 
\[
  2\ol \Ga_{\ddot s}(Ld(x_i)) \ge 2\ul \Ga_{\ddot s}(Ld(x_i)) \ge 2\ul \Ga_{\ddot s}(3L d(x)) 
  \ge  2\ol \Ga_{\dddot s}(3LDd(x))
\]
and analogously for $2\ol \Ga_{\ddot s}(Ld(x))$.
Thus, by the calculation in \eqref{calculation}, 
\begin{align*}
  H_2 
  &\le \frac{C |\be|!}{(n d(x))^{|\be|}} \sum_{j =2 q }^{\ul \Ga_s(L \la  d(x))}
  (2n^2 \rh  d(x))^{j} s_{j}. 
\end{align*}
By \Cref{basic}\eqref{eq:ulGa3}, 
\[s_j (L \la  d(x))^j \le s_{2 q} (L \la  d(x))^{2 q} \quad \text{ for } 2q \le j \le \ul \Ga_s(L \la  d(x)). \]
By \eqref{eq:mge1}, $s_{2q}   \le 
  \dot s_{q}^2 \le \dddot s_{q}^2$. By the definition of $q$,
\[
  h_{\dddot s}(3LD d(x)) = \dddot s_{q} (3LD d(x))^{q} \le \dddot s_{|\be|} (3LD d(x))^{|\be|} \quad \text{ for all } 
  |\be|.
\]
All this leads to 
\begin{align*}
  H_2 
  &\le \frac{C |\be|!}{(n d(x))^{|\be|}} \sum_{j =2 q }^{\ul \Ga_s(L \la  d(x))}
  \Big(\frac{2n^2 \rh}{L\la } \Big)^{j} s_{2 q} (L \la  d(x))^{2 q}
  \\
  &\le \frac{C |\be|!}{(n d(x))^{|\be|}} \sum_{j =2 q }^{\ul \Ga_s(L \la  d(x))}
  \Big(\frac{2n^2 \rh}{L\la } \Big)^{j} \dddot s_{q}^2 ( L \la  d(x))^{2 q}
  \\
  &\le  C \Big(\frac{3LD}{n}\Big)^{|\be|} |\be|!\, \dddot s_{|\be|}  h_{\dddot s}(3LD d(x)) 
  \Big(\frac{\la }{3D}\Big)^{2 q}
  \sum_{j =2 q }^{\ul \Ga_s(L \la  d(x))}
  \Big(\frac{2n^2  \rh}{L\la } \Big)^{j}.   
\end{align*}
If we choose $L \ge 4n^2  \rh/\la$, then the sum is bounded by $2$, 
and \eqref{H2} follows, as $\la<1$ and $D \ge 1$. 
\end{proof}

\subsection{The extension theorem} \label{sec:extproof}

\begin{theorem}\label{extensiontheorem}
Let $\om$ be a non-quasianalytic concave weight function
and let $\si$ be a good heir of $\om$.
Let $E$ be a compact subset of $\R^n$. 
Then the jet mapping $j^\infty_E : \cB^{\{\om\}}(\R^n) \to \cB^{\{\si\}}(E)$ is surjective.
\end{theorem}

\begin{proof}
	We assume that the setup of \Cref{sec:prelimII} holds.		
	Assume 
	\begin{equation} \label{eq:L}
		L > \max\{C_1,C_2\} \, \rh	
	\end{equation}
	so that \eqref{H1} and \eqref{H2} are valid.

Let $p\in \N$ be fixed (and to be specified later). 
Let $\{\vh_{i,p}\}_{i\in \N}$ be the family of functions provided by \Cref{Proposition6}, relative to the 
family of cubes $\{Q_i\}_{i \in \N}$ from \Cref{Whitneycubes},
and let $r_0 = r_0(p)$ be the constant appearing in this proposition. 
Recall that $x_i$ denotes the center of $Q_i$.

We will show that an extension of class $\cB^{\{\om\}}$ of $F$ to $\R^n$ is provided by
\[
  f(x) := 
  \begin{cases}
    \sum_{i\in \N} \vh_{i,p}(x) \, T_{\hat x_i}^{2 \ol \Ga_{\ddot s}(L d(x_i)) } F(x),  & \text{ if } x \in \R^n \setminus E, \\
    F^0(x), & \text{ if } x \in  E.
  \end{cases}
\]
Clearly, $f$ is $C^\infty$ in $\R^n \setminus E$. 

In the following $\fW$ denotes the weight matrix associated with $\om$.

\begin{claim*} 
There exist constants $K_j=K_j(S)$, $j = 1,2,3$, such that the following holds.
  If $p = K_1 L$ and $L>K_2 \rh$, then there exist	
  weight sequences $W \in \fW$, $\tilde S \in \fS$ and a constant $M_1= M_1(S,L)>0$ such that 
	for all $x \in \R^n \setminus E$ with $d(x) <  (3B_1)^{-1} r_0$ and all $\al \in \N^n$,
  \begin{equation} \label{eqclaim2}
    |\p^\al (f - T_{\hat x}^{2 \ol \Ga_{\ddot s}(L d(x)) } F) (x)| \le C M_1^{|\al|+1}  W_{|\al|} h_{\tilde s}(LK_3 d(x)), 
  \end{equation}
  where $C$ and $\rh$ are the constants from \eqref{jets1} and \eqref{jets2} (and $B_1$ stems from \Cref{Whitneycubes}).
\end{claim*}

\paragraph{\emph{Proof of the claim}}

By the Leibniz rule,
\begin{align}
  \p^\al& (f - T_{\hat x}^{2 \ol \Ga_{\ddot s}(L d(x)) } F) (x) \notag
  \\
  &=
  \sum_{\be \le \al} \binom{\al}{\be} \sum_i \vh_{i,p}^{(\al-\be)}(x) \, 
  \p^\be (T_{\hat x_i}^{2 \ol \Ga_{\ddot s}(L d(x_i)) } F - T_{\hat x}^{2 \ol \Ga_{\ddot s}(L d(x)) } F) (x). \label{Leibniz}
\end{align}
Using \Cref{cor:Whitneycubes}, $\ddot s \le \dddot s$ which entails $h_{\ddot s} \le h_{\dddot s}$, and the fact that $h_{\dddot s}$ is increasing, we conclude from
\eqref{H1} and \eqref{H2}, that for $x \in Q_i^*$ with $d(x)<1$,
\begin{align*}
  |\p^\be (T_{\hat x_i}^{2 \ol \Ga_{\ddot s}(L d(x_i)) } F - T_{\hat x}^{2 \ol \Ga_{\ddot s}(L d(x)) } F) (x)|
  \le C    
    (6D L)^{|\be|+1}  \dddot S_{|\be|} \, h_{\dddot s }(3LD d(x)).
\end{align*}

By \Cref{Proposition6}, 
there exist $W = W(p) \in \fW$ and $M = M(p)>0$ such 
that, provided that $d(Q_i,E) \le r_0/B_1$, we have,
for all $\be \in \N^n$ and $x\in\R^n\setminus E$,
\begin{align*}
  |\varphi^{(\be)}_{i,p}(x)| &\le M W_{|\be|} 
  \exp\Big(\frac{A_1(n)}{p}\si^{\star}\Big(\frac{b_1 p}{A_2(n)} \on{diam} Q_i\Big)\Big)
  \\
  &\le M W_{|\be|} 
  \exp\Big(\frac{A_1(n)}{p}\si^{\star}\Big(\frac{b_1 p}{9 A_2(n)} d(x)\Big)\Big) 
  =: M W_{|\be|} \Pi(p,x),
\end{align*}
by \Cref{cor:Whitneycubes},
since $\si^\star$ is decreasing (recall that $\supp(\vh_{i,p})  \subseteq Q_i^*$). 

Let us assume that $x \in \R^n \setminus E$ satisfies $d(x) <  (3B_1)^{-1} r_0$. Then, if $x \in Q_i^*$, 
\[
d(Q_i,E) \le d(x_i) \le  3d(x) \le \frac{r_0}{B_1},
\]
by \Cref{cor:Whitneycubes}.
So, for all $i \in \N$, $x \in \R^n \setminus E$ with $d(x) <  (3B_1)^{-1} r_0$, and $\be \in \N^n$,
\[
  |\varphi^{(\be)}_{i,p}(x)|\le  M W_{|\be|}\, \Pi(p,x).
\]
By \Cref{lem:heir} and \Cref{lemma4}\eqref{lemma4(2)}, we may assume that $\dddot S \le D_1\, W$ for some constant $D_1$.
Then, by \eqref{Leibniz} and \Cref{Whitneycubes}, 
for $x \in \R^n \setminus E$ with $d(x) <  (3B_1)^{-1} r_0$,  
\begin{align*}
  |\p^\al& (f - T_{\hat x}^{2 \ol \Ga_{\ddot s}(L d(x)) } F) (x)| 
  \\
  &\le
  \sum_{\be \le \al} \frac{\al!}{\be!(\al-\be)!} 
  \cdot 12^{2n} \cdot M W_{|\al|-|\be|}\Pi(p,x)
  \cdot C    (6DL)^{|\be|+1}  \dddot S_{|\be|} \, h_{\dddot s}(3LD d(x))
  \\
  &\le  12^{2n}C  M
  \Big(\sum_{j=0}^{|\al|} \frac{|\al|!\, n^{|\al|+j}}{j!(|\al|-j)!}  
    (6DL)^{j+1} W_{|\al|-j}
        \dddot S_{j}\Big) \, \Pi(p,x)
  \, h_{\dddot s}(3LD d(x))    
  \\
  &\le 6\, 12^{2n} DL CD_1  M n^{|\al|} W_{|\al|}  
  \Big(\sum_{j=0}^{|\al|} \frac{|\al|!\, }{j!(|\al|-j)!}   
      (6D L n)^{j}\Big) \, 
      \Pi(p,x)
  \, h_{\dddot s}(3LD d(x))
  \\
  &=  6\, 12^{2n} DL CD_1  M (n (1 + 6DL n ))^{|\al|} W_{|\al|}  
  \Pi(p,x)
  \, h_{\dddot s}(3LD d(x)),           
\end{align*}
since $W_{|\al|-j} \dddot S_j \le D_1\, W_{|\al|-j} W_j \le D_1\, W_{|\al|}$, by \Cref{lem:basicM}(2). 
By \eqref{eq:key1}, for each $\tilde S \in \fS$ there is a constant $H\ge 1$ such that   
\[
	\Pi(p,x) 
	\le \Big(\frac{e}{h_{\tilde s}(\frac{b_1 p d(x)}{9 A_2(n) H})} \Big)^{\frac{A_1(n) H}{p}}.
\]
By \Cref{lem:hmodgrowth}, there is a constant $B\ge 1$ such that $h_{\dddot s}(t) \le h_{\tilde s}(Bt)^2$ provided that 
$\dddot S_{j+k} \le  \tilde S_j \tilde S_k$ for all $j,k$. That such $\tilde S \in \fS$ exists follows from 
\Cref{lemma4}\eqref{eq:mW}.  

Let us choose $L$ according to \eqref{eq:L} and such that $p := 27 \,A_2(n) H B  D L/b_1 \ge A_1(n) H$ is an integer. 
Then,  since $h_{\tilde s} \le 1$, 
\begin{align*}
	\Pi(p,x) \, h_{\dddot s}(3LD d(x))
	\le \frac{e h_{\dddot s}(3LD d(x))}{h_{\tilde s}(3  B  L D   d(x))} 
	\le e h_{\tilde s}(3  B  L D   d(x))
\end{align*}
 and we obtain \eqref{eqclaim2}. 
 (Note that $M$ depends on $p$, and hence on $L$, which results in the non-explicit dependence of $M_1$.)
 The claim is proved.

\medskip

\paragraph{\emph{End of proof}}

By \eqref{eq:ext1}, we have \eqref{eq:mass1} with $S' := \ddot S$.
We may additionally assume that $L \ge D_1 \rh$ for the corresponding constant $D_1$ in \Cref{proposition9}. 
So, by \eqref{prop91} and 
 \eqref{eqclaim2}, for $x \in \R^n \setminus E$ with $d(x) <  (3B_1)^{-1} r_0$ and $\al \in \N^n$, 
 \begin{align}
    |f^{(\al)}(x)| 
    &\le 
   |(T_{\hat x}^{2 \ol \Ga_{\ddot s}(L d(x))} F)^{(\al)}(x)| + |\p^\al (f - T_{\hat x}^{2 \ol \Ga_{\ddot s}(L d(x)) } F) (x)|
   \notag \\ \label{final}
   &\le 
    C M^{|\al|+1} W_{|\al|}
 \end{align}
 for a suitable constant $M=M(S,L)$; here we use that $h_{\tilde s} \le 1$.

 Let us fix a point $a \in E$ and $\al\in \N^n$. 
 Since $\ol \Ga_{\ddot s}(t) \to \infty$ as $t \to 0$ (see \Cref{basic}), 
 we have $|\al| < 2 \ol \Ga_{\ddot s}(L d(x))$ if $x \in \R^n \setminus E$ is 
 sufficiently close to $a$. Thus, as $x \to a$, 
 \begin{align*}
   &|f^{(\al)}(x) - F^{\al}(a)|
   \\
   &\le 
   |\p^\al (f - T_{\hat x}^{2 \ol \Ga_{\ddot s}(L d(x)) } F) (x)| + 
   |(T_{\hat x}^{2 \ol \Ga_{\ddot s}(L d(x))}F)^{(\al)}(x)-F^\al(\hat x)| + |F^\al(\hat x) - F^\al(a)| 
   \\
   & = O(h_{\tilde s}(L K_3 d(x))) + O(d(x)) + O(|\hat x - a|),
 \end{align*}
 by \eqref{jets2}, \eqref{prop92} (where $S' =\ddot S$), and \eqref{eqclaim2}. 
 Hence $f^{(\al)}(x) \to  F^{\al}(a)$ as $x \to a$.
 We may conclude that $f \in C^\infty(\R^n)$.
 After multiplication with a suitable cut-off function of class $\cB^{\{\om\}}$ with support in 
 $\{x : d(x) <  (3B_1)^{-1} r_0\}$,
 we find that $f \in \cB^{\{\om\}}(\R^n)$ thanks to \eqref{jets1}, \eqref{final}, and \Cref{lemma4}\eqref{5.10}.
 The theorem is proved.
\end{proof}

\begin{remark}
  The proof of Theorem \ref{extensiontheorem} shows that for each $\rh>0$ there exist $M(\rh)>0$ and 
  a continuous linear extension operator 
  $\cB^S_\rh(E) \to \cB^{W}_{M(\rh)}(\R^n)$. 
  This extension operator depends on 
  $\rh$ and $S$ (through $L$ and $p$) and in general there is no continuous extension operator 
  $\cB^{\{\si\}}(E) \to \cB^{\{\om\}}(\R^n)$, 
  cf.\ \cite{Petzsche88} and 
  \cite[p.\,223]{schmetsValdivia00}. 
\end{remark}

\subsection{The extension theorem for Denjoy--Carleman classes} \label{sec:CC}

In this section we prove a consequence of the extension theorem \ref{extensiontheorem}
for Denjoy--Carleman classes and compare it with the result of \cite{ChaumatChollet94}. 
The \emph{sine qua non} for the extension of jets of class $\cB^{\{M\}}$ to a function of class $\cB^{\{N\}}$ 
is the following condition:
\begin{equation} \label{eq:CC}
    \E C >0 \A k \in \N: \sum_{\ell \ge k} \frac{1}{\nu_\ell} \le C\, \frac{k}{\mu_k}.
\end{equation} 
(This is true, if $M$ has moderate growth, 
which we shall have to assume in the main result of this section, \Cref{thm:sequences}. 
In general, the right condition seems to be ($*$) in \cite{SchmetsValdivia04} which is equivalent to \eqref{eq:CC} 
provided that $M$ has moderate growth, see \cite[2.(c)]{SchmetsValdivia04}.)
In the next lemma we show that \eqref{eq:CC} implies \eqref{eq:heir} for the associated weight functions 
$\om_M$ and $\om_N$. This is based on \cite[Proposition 4.4]{Komatsu73} and was announced in 
\cite[p.\,39]{ChaumatChollet94}. We include a full proof for the convenience of the reader.

\begin{lemma} \label{lem:CCstrong}
  Let $M$ and $N$ be weight sequences satisfying $\mu \le \nu$ and \eqref{eq:CC}.
  Then
  \begin{equation} \label{eq:strong}
    \E C >0 \A t >0 : \int_1^\infty \frac{\om_N(t u)}{u^2}\, du \le C \om_M(t) + C. 
  \end{equation}

\end{lemma}

\begin{proof}
  Let $\Si_M(t) := \max\{ k : \mu_k \le t\}$. 
  Then (cf.\ \cite[p.\,21]{Mandelbrojt52}) 
  \begin{equation} \label{eq:Siom}
    \om_M(t) = \int_0^t \frac{\Si_M(u)}{u}\, du.
  \end{equation}
  Similarly for $\Si_N$ and $\om_N$. 
  Since $\mu_k \le \nu_k$ for all $k$ we have $\Si_N \le \Si_M$. 
  By \eqref{eq:CC}, $N$ is non-quasianalytic and, by \cite[Lemma 4.1]{Komatsu73},  
  \[
    \frac{\Si_N(t)}{t} \to 0 \quad \text{ and } \quad  \frac{\om_N(t)}{t} \to 0 \quad \text{ as } \quad t \to \infty.
  \]
  Fix $t \ge \nu_1$ and
  set $p := \Si_M(t)$ and $q := \Si_N(t)$; then $p \ge q \ge 1$, 
  $\nu_{q+1}>t$, and $\mu_{p+1}>t$.
  Integration by parts yields 
  \begin{align*}
    \int_t^\infty \frac{\Si_N(u)}{u^2}\, du 
    &= \frac{\Si_N(t)}{t} + \int_t^\infty \frac{d \Si_N(u) }{u} 
    \\
    &= \frac{\Si_N(t)}{t} + \sum_{\ell = q+1}^\infty \frac{1}{\nu_{\ell}}
    \\
    &= \frac{\Si_N(t)}{t} + \sum_{\ell = q+1}^p \frac{1}{\nu_{\ell}} + \sum_{\ell = p+1}^\infty \frac{1}{\nu_{\ell}}
    \\
    &\le \frac{\Si_N(t)}{t} +  \frac{p}{\nu_{q+1}} + C\, \frac{p+1}{\mu_{p+1}}
    \\
    &\le (2 + 2 C) \frac{\Si_M(t)}{t} \quad \text{ for } t \ge \nu_1.
  \end{align*}
  Consequently, by integrating,
  \begin{align*}
    \int_{\nu_1}^s \int_t^\infty \frac{\Si_N(u)}{u^2}\, du\, dt 
    &= 
    s \int_s^\infty \frac{\Si_N(u)}{u^2}\, du  + \int_{\nu_1}^s \frac{\Si_N(u)}{u}\, du 
    - \nu_1 \int_{\nu_1}^\infty \frac{\Si_N(u)}{u^2}\, du
    \\
    &\le 
    (2 + 2 C) \int_{\nu_1}^s \frac{\Si_M(t)}{t}\, dt
    \\
    &\le 
    (2 + 2 C) \om_M(s) \quad \text{ for } s \ge \nu_1.
  \end{align*}
  It follows that, for $s \ge \nu_1$,  
  \begin{align*}
    s \int_s^\infty \frac{\Si_N(u)}{u^2}\, du  \le (2 + 2 C) \om_M(s) + \nu_1 \int_0^\infty \frac{\Si_N(u)}{u^2}\, du
  \end{align*}
  Clearly, this also holds for all $0 <s<\nu_1$. 
  By partial integration and \eqref{eq:Siom},
  \begin{equation*}
    \int_s^\infty \frac{\om_N(u)}{u^2} \, du  = \frac{\om_N(s)}{s} + \int_s^\infty \frac{\Si_N(u)}{u^2}\, du.
  \end{equation*}
  Thus, using $\om_N \le \om_M$,
  \begin{align*}
    s \int_s^\infty \frac{\om_N(u)}{u^2}\, du  \le (3 + 2 C) \om_M(s) + \nu_1 \int_0^\infty \frac{\Si_N(u)}{u^2}\, du
  \end{align*}
  which implies \eqref{eq:strong}.
\end{proof}

\begin{lemma} \label{lem:MomM}
  Let $M$ be a weight sequence of moderate growth
  such that $\liminf_{k \to \infty} m_k^{1/k}>0$ and $\liminf_{k \to \infty} \mu_{Qk}/\mu_k >1$ 
  for some $Q \in \N_{\ge 2}$.   
  Then $\cB^{\{M\}}(\R^n) = \cB^{\{\om_M\}}(\R^n)$ and $\cB^{\{M\}}(E) = \cB^{\{\om_M\}}(E)$ for each compact 
  $E \subseteq \R^n$. 
\end{lemma}

\begin{proof}
  Since $M$ is log-convex, we have (cf.\ e.g.\ \cite{Mandelbrojt52})
  \begin{equation} \label{eq:MW1}
    M_k = \sup_{t>0} \frac{t^k}{e^{\om_M(t)}} = \sup_{s \in \R} \frac{e^{sk}}{e^{\vh_{\om_M}(s)}} = e^{\vh_{\om_M}^*(k)}.
  \end{equation}
  So the first identity follows from \Cref{representation}, since $\om_M$ is a weight function. 
  The second identity is a consequence of 
  \cite[(5.11)]{RainerSchindl12} and \Cref{def:ultrajets}. 
\end{proof} 

\begin{theorem} \label{thm:sequences}
  Let $M$ and $N$ be weight sequences of moderate growth  
  satisfying $\mu \lesssim \nu$.  
  Assume that both the sequences $\mu_k/k$ and $\nu_k/k$ tend to infinity and are almost increasing in the sense that 
  \begin{equation} \label{eq:condai}
    \E C >0 \A 1 \le j \le k : \frac{\mu_j}{j} \le C \, \frac{\mu_k}{k}.
  \end{equation}
  Then the following conditions are equivalent:
  \begin{enumerate}
    \item For every compact $E \subseteq \R^n$ the jet mapping $j^\infty_E : \cB^{\{N\}}(\R^n) \to \cB^{\{M\}}(E)$ 
    is surjective. 
    \item There is a $C>0$ such that $\int_{1}^\infty \frac{\om_N(tu)}{u^2}\,du \le C\,\om_M(t) + C$ for all $t>0$.
    \item There is a $C>0$ such that $\sum_{\ell \ge k} \frac{1}{\nu_\ell} \le C\, \frac{k}{\mu_k}$ for all $k \in \N$.
  \end{enumerate} 
\end{theorem}

\begin{proof}
  (3) $\Rightarrow$ (2) follows from \Cref{lem:CCstrong} since we may assume without loss of generality that $\mu \le \nu$ 
  (otherwise we replace $(N_k)$ by an equivalent sequence $(C^k N_k)$).

  (2) $\Rightarrow$ (1) Since $M$ has moderate growth, $\mu_k/k \lesssim m_k^{1/k}$ tends to infinity and hence 
  $\om_M(t) = o(t)$ as $t \to \infty$. By (2), $\om_M$ is a heir of $\om_N$.
  The condition \eqref{eq:condai} (and \eqref{eq:MW1}) guarantees that $\om_M$ is a good heir of $\om_N$. 
  Moreover, \eqref{eq:condai} for $\nu$ implies that $\om_N$ is equivalent to its least concave majorant; 
  this follows from \Cref{thm:subadditive} and \Cref{thm:suffgood} since $\nu_k \lesssim N_k^{1/k}$ as $N$ has moderate growth 
  (cf.\ \cite[Lemma 2.2]{RainerSchindl16a}). 
  So \Cref{extensiontheorem} and \Cref{lem:MomM} entail (1).

  (1) $\Rightarrow$ (3) This follows from \cite[Proposition 27]{ChaumatChollet94}; an inspection of its proof 
  shows that the general assumption of \cite{ChaumatChollet94} that all sequences are \emph{strongly log-convex} 
  (i.e.\ $m_k = M_k/k!$ is log-convex)
  is not needed.
  Alternatively, it is a consequence of \cite[Theorem 1.1]{SchmetsValdivia04} thanks to \cite[2.(c)]{SchmetsValdivia04}.  
\end{proof}

\begin{remark}
  By \Cref{thm:suffgood},
  the condition \eqref{eq:condai} can be replaced by 
  \begin{equation} \label{eq:condai2}
    \E C >0 \A 1 \le j \le k : m_j^{1/j} \le C \, m_k^{1/k}.
  \end{equation}
  \Cref{thm:sequences} should be compared with \cite[Theorem 30]{ChaumatChollet94}. 
  In the latter the sequences $M$ and $N$ are assumed to be strongly log-convex which entails the weaker condition 
  \eqref{eq:condai} (and \eqref{eq:condai2}). On the other hand, in \cite[Theorem 30]{ChaumatChollet94} moderate growth 
  of $N$ is not required.
\end{remark}

\subsection{Strong and good weight functions} \label{sec:strong}

In view of \Cref{q:concave} and the fact that every strong weight function is equivalent to a concave one it is 
natural to ask:

\begin{question} \label{q:strong}
  Is every strong weight function equivalent to a good one?
\end{question}

By \Cref{thm:suffgood}, this holds true if the associated weight matrix satisfies \eqref{eq:m5}. 
We do not know the general answer to this question. However, we can provide some more information on strong weight functions.
We start with a corollary to \Cref{lem:CCstrong}.

\begin{corollary} \label{cor:weakcond}
  Let $\om$ be a weight function and let $\fW = \{W^x\}_{x>0}$ be the associated weight matrix.
  Then $\om$ is a strong weight function provided that 
  \begin{equation} \label{eq:weakcond}
    \E x,y>0 : \sum_{\ell \ge k} \frac{1}{\vt^y_\ell} \lesssim \frac{k}{\vt^x_k}.
  \end{equation}
\end{corollary}

\begin{proof}
  This is immediate from \Cref{lem:CCstrong} since $\om$ is equivalent to $\om_{W^x}$ for all $x>0$, 
  by \cite[Lemma 5.7]{RainerSchindl12}, and \eqref{eq:strong} is invariant under equivalence.
\end{proof}

\begin{remark} \label{rem:weakcond}
  Together with \Cref{thm:preservingclass} and \cite[Corollary 5.13]{RainerSchindl16a}, the corollary implies that 
  actually \eqref{eq:weakcond} is equivalent to $\om$ being strong, provided that the associated weight matrix 
  satisfies \eqref{eq:m5}. We do not know if this is always true.   
\end{remark}

The next lemma is based 
on a construction from \cite{RainerSchindl16a} 
which stems from an idea in \cite[Proposition 1.1]{Petzsche88}.

\begin{lemma} \label{lem:strongmatrix}
 Let $\fM = \{M^x\}_{x>0}$ be a collection of non-quasianalytic weight sequences satisfying $\mu^x \lesssim \mu^y$ whenever 
 $x \le y$. Assume that 
 \begin{equation}
    \A x >0 \E y>0 \E C>0 : \sum_{\ell \ge k} \frac{1}{\mu^y_\ell} \le C \, \frac{k}{\mu^x_k}.
  \end{equation}
  Then there exists a collection of non-quasianalytic weight sequences $\fS = \{S^x\}_{x>0}$ 
  with the following properties:
  \begin{enumerate}
     \item $1\le \si^x_k/k$ is increasing to $\infty$ for all $x>0$.
     \item $\si^x \lesssim \si^y$ whenever $x \le y$.
     \item $\A x >0 \E y>0 \E C>0 : \sum_{\ell \ge k} \frac{1}{\si^y_\ell} \le C \, \frac{k}{\si^x_k}$.
     \item $\A x >0 \E y>0 : \si^x \lesssim \mu^y$ and $\A x >0 \E y>0 : \mu^x \lesssim \si^y$. 
   \end{enumerate} 
\end{lemma}

\begin{proof}
With any positive increasing sequence
$\mu=(\mu_k)$ satisfying $\mu_0 =1$ and $\sum_k 1/\mu_k < \infty$ we may associate a 
positive sequence $\si = \si(\mu)$ in the following way: 
we define 
  \begin{equation*} 
    \ta_k := \frac{k}{\mu_k} + \sum_{j\ge k} \frac 1 {\mu_j}, \quad k \ge 1,
  \end{equation*}
and set 
  \begin{equation*}     
    \si_k := \frac{\ta_1 k}{\ta_k}, \quad k \ge 1, \quad \si_0 := 1.  
  \end{equation*}
Then, cf.\ \cite[Lemma 4.2]{RainerSchindl16a},
  \begin{itemize}
    \item $\si \lesssim \mu$.
    \item $\sum_{j\ge k}  1 /\mu_j \lesssim k/\si_k$.
    \item $1 \le  \si_k/k$ is increasing to $\infty$ (in particular, $S$ is strongly log-convex). 
    \item If $\mu'$ is an increasing positive sequence satisfying $\mu'  \lesssim \mu$ 
    and $\sum_{j\ge k}  1 /\mu_j \lesssim k/\mu'_k$, then $\mu' \lesssim \si$.  
  \end{itemize} 
If we apply this construction to the sequences in $\fM$ we obtain a collection of weight sequences $\fS = \{S^x\}_{x>0}$ 
which satisfies the properties (1)--(4).   
By (4), there exists $x_0>0$ such that $S^{x}$ is non-quasianalytic for all $x\ge x_0$. 
If we set $S^x := S^{x_0}$ for all $x<x_0$, the collection $\fS$ is as desired.  
\end{proof}

\begin{theorem}[Strong weight functions]
  Let $\om$ be a weight function.
  Assume that the associated weight matrix $\fW= \{W^x\}_{x >0}$ satisfies
  \begin{equation} \label{gmg} 
    \A x>0 \E y>0  : \vt_k^x \lesssim (W^y_k)^{1/k}. 
  \end{equation}
  Then the following conditions are equivalent:
  \begin{enumerate}
    \item $\om$ is strong.
    \item $\forall x >0 \E y>0 \E C>0 : \sum_{\ell \ge k} \frac{1}{\vt^y_\ell} \le C \, \frac{k}{\vt^x_k}$.
    \item $\exists x >0 \E y>0 \E C>0 : \sum_{\ell \ge k} \frac{1}{\vt^y_\ell} \le C \, \frac{k}{\vt^x_k}$.
  \end{enumerate}
  If $\om$ is strong, then there exists a collection $\fS = \{S^x\}_{x>0}$ of strongly log-convex sequences such that 
  \begin{equation} \label{eq:sivt}
    \A x >0 \E y>0 : \si^x \lesssim \vt^y \quad \text{ and } \quad \A x >0 \E y>0 : \vt^x \lesssim \si^y
  \end{equation}
  and $\om$ is good.
\end{theorem}

Note that \eqref{gmg} is only needed for $(1) \Rightarrow (2)$.

\begin{proof}
  For the equivalence of (1), (2), and (3), see \Cref{cor:weakcond}, \Cref{rem:weakcond}, and the references cited therein.

  \Cref{lem:strongmatrix} implies the statement about $\fS$. 
  Goodness of $\om$ follows either from \Cref{thm:suffgood}, since a strong weight function is equivalent to a concave one, or 
  from the strong log-convexity of the $S^x$ and \eqref{eq:sivt}: for each $x>0$ we find $y,z>0$ such that
  \[
  \frac{\vt^x_j}{j} \lesssim \frac{\si^y_j}{j} \le \frac{\si^y_k}{k} \lesssim \frac{\vt^z_k}{k}, 
  \quad \text{ for } 1 \le j \le k.  \qedhere
  \]
\end{proof}

We remark that the condition \eqref{eq:sivt} entails $\cB^{\{\om\}}(\R^n) = \on{ind}_{x>0} \cB^{\{S^x\}}(\R^n)$, 
by \Cref{representation}.

\def\cprime{$'$}
\providecommand{\bysame}{\leavevmode\hbox to3em{\hrulefill}\thinspace}
\providecommand{\MR}{\relax\ifhmode\unskip\space\fi MR }
\providecommand{\MRhref}[2]{%
  \href{http://www.ams.org/mathscinet-getitem?mr=#1}{#2}
}
\providecommand{\href}[2]{#2}


\begin{thebibliography}{10}

\bibitem{Abanin01}
A.~V. Abanin, \emph{On {W}hitney's extension theorem for spaces of
  ultradifferentiable functions}, Math. Ann. \textbf{320} (2001), no.~1,
  115--126. 

\bibitem{Beurling61}
A.~Beurling, \emph{Quasi-analyticity and general distributions}, Lecture notes,
  AMS Summer Institute, Stanford, 1961.

\bibitem{Bjoerck66}
G.~Bj{\"o}rck, \emph{Linear partial differential operators and generalized
  distributions}, Ark. Mat. \textbf{6} (1966), 351--407.

\bibitem{BBMT91}
J.~Bonet, R.~W. Braun, R.~Meise, and B.~A. Taylor, \emph{Whitney's extension
  theorem for nonquasianalytic classes of ultradifferentiable functions},
  Studia Math. \textbf{99} (1991), no.~2, 155--184.

\bibitem{BMM07}
J.~Bonet, R.~Meise, and S.~N. Melikhov, \emph{A comparison of two different
  ways to define classes of ultradifferentiable functions}, Bull. Belg. Math.
  Soc. Simon Stevin \textbf{14} (2007), 424--444.

\bibitem{BonetMeiseTaylor89a}
J.~Bonet, R.~Meise, and B.~A. Taylor, \emph{Whitney's extension theorem for
  ultradifferentiable functions of {R}oumieu type}, Proc. Roy. Irish Acad.
  Sect. A \textbf{89} (1989), no.~1, 53--66. 

\bibitem{BonetMeiseTaylor92}
\bysame, \emph{On the range of the {B}orel map for classes of nonquasianalytic
  functions}, Progress in functional analysis ({P}e\~niscola, 1990),
  North-Holland Math. Stud., vol. 170, North-Holland, Amsterdam, 1992,
  pp.~97--111. 

\bibitem{BMT90}
R.~W. Braun, R.~Meise, and B.~A. Taylor, \emph{Ultradifferentiable functions
  and {F}ourier analysis}, Results Math. \textbf{17} (1990), no.~3-4, 206--237.

\bibitem{Bruna80}
J.~Bruna, \emph{An extension theorem of {W}hitney type for non-quasi-analytic
  classes of functions}, J. London Math. Soc. (2) \textbf{22} (1980), no.~3,
  495--505.

\bibitem{ChaumatChollet94}
J.~Chaumat and A.-M. Chollet, \emph{Surjectivit\'e de l'application restriction
  \`a un compact dans des classes de fonctions ultradiff\'erentiables}, Math.
  Ann. \textbf{298} (1994), no.~1, 7--40.

\bibitem{Debrouwere14}
A.~Debrouwere, \emph{Analytic representations of distributions and
  ultradistributions}, Master's thesis, Ghent University, 2014,
  \url{http://lib.ugent.be/fulltxt/RUG01/002/163/702/RUG01-002163702_2014_0001_AC.pdf}.

\bibitem{Dynkin80}
E.~M. Dyn'kin, \emph{{Pseudoanalytic extension of smooth functions. The uniform
  scale.}}, Transl., Ser. 2, Am. Math. Soc. \textbf{115} (1980), 33--58
  (English).

\bibitem{FernandezGalbis06}
C.~Fern{\'a}ndez and A.~Galbis, \emph{Superposition in classes of
  ultradifferentiable functions}, Publ. Res. Inst. Math. Sci. \textbf{42}
  (2006), no.~2, 399--419.

\bibitem{Hoermander83I}
L.~H{\"o}rmander, \emph{The analysis of linear partial differential operators.
  {I}}, Grundlehren der Mathematischen Wissenschaften [Fundamental Principles
  of Mathematical Sciences], vol. 256, Springer-Verlag, Berlin, 1983,
  Distribution theory and Fourier analysis.

\bibitem{Komatsu73}
H.~Komatsu, \emph{Ultradistributions. {I}. {S}tructure theorems and a
  characterization}, J. Fac. Sci. Univ. Tokyo Sect. IA Math. \textbf{20}
  (1973), 25--105.

\bibitem{Lambert79}
A.~Lambert, \emph{Quelques th{\'e}or{\`e}mes de d{\'e}composition des
  ultradistributions}, Ann. Inst. Fourier (Grenoble) \textbf{29} (1979), no.~3,
  x, 57--100. 

\bibitem{Langenbruch94}
M.~Langenbruch, \emph{Extension of ultradifferentiable functions}, Manuscripta
  Math. \textbf{83} (1994), no.~2, 123--143. 

\bibitem{Malgrange67}
B.~Malgrange, \emph{Ideals of differentiable functions}, Tata Institute of
  Fundamental Research Studies in Mathematics, No. 3, Tata Institute of
  Fundamental Research, Bombay, 1967.

\bibitem{Mandelbrojt52}
S.~Mandelbrojt, \emph{S\'eries adh\'erentes, r\'egularisation des suites,
  applications}, Gauthier-Villars, Paris, 1952.

\bibitem{MeiseTaylor88}
R.~Meise and B.~A. Taylor, \emph{Whitney's extension theorem for
  ultradifferentiable functions of {B}eurling type}, Ark. Mat. \textbf{26}
  (1988), no.~2, 265--287.

\bibitem{Peetre70}
J.~Peetre, \emph{Concave majorants of positive functions}, Acta Math. Acad.
  Sci. Hungar. \textbf{21} (1970), 327--333.

\bibitem{Petzsche88}
H.-J. Petzsche, \emph{On {E}. {B}orel's theorem}, Math. Ann. \textbf{282}
  (1988), no.~2, 299--313.

\bibitem{PetzscheVogt84}
H.-J. Petzsche and D.~Vogt, \emph{Almost analytic extension of
  ultradifferentiable functions and the boundary values of holomorphic
  functions}, Math. Ann. \textbf{267} (1984), no.~1, 17--35.

\bibitem{RainerSchindl12}
A.~Rainer and G.~Schindl, \emph{Composition in ultradifferentiable classes},
  Studia Math. \textbf{224} (2014), no.~2, 97--131.

\bibitem{RainerSchindl14}
\bysame, \emph{Equivalence of stability properties for ultradifferentiable
  function classes}, Rev. R. Acad. Cienc. Exactas Fis. Nat. Ser. A Math.
  RACSAM. \textbf{110} (2016), no.~1, 17--32.

\bibitem{RainerSchindl16a}
\bysame, \emph{Extension of {W}hitney jets of controlled growth}, 
  Math. Nachr. \textbf{290} (2017), no.~14-15, 2356--2374. 

\bibitem{schmetsValdivia00}
J.~Schmets and M.~Valdivia, \emph{Extension maps in ultradifferentiable and
  ultraholomorphic function spaces}, Studia Math. \textbf{143} (2000), no.~3,
  221--250. 

\bibitem{SchmetsValdivia04}
\bysame, \emph{On certain extension theorems in the mixed {B}orel setting}, J.
  Math. Anal. Appl. \textbf{297} (2004), no.~2, 384--403, Special issue
  dedicated to John Horv{{\'a}}th. 

\bibitem{Stein70}
E.~M. Stein, \emph{Singular integrals and differentiability properties of
  functions}, Princeton Mathematical Series, No. 30, Princeton University
  Press, Princeton, N.J., 1970. 

\bibitem{Tougeron72}
J.-C. Tougeron, \emph{Id\'eaux de fonctions diff\'erentiables},
  Springer-Verlag, Berlin, 1972, Ergebnisse der Mathematik und ihrer
  Grenzgebiete, Band 71.

\bibitem{Whitney34a}
H.~Whitney, \emph{Analytic extensions of differentiable functions defined in
  closed sets}, Trans. Amer. Math. Soc. \textbf{36} (1934), no.~1, 63--89.

\end{thebibliography}

\end{document}